\documentclass[article,12pt]{amsart}
\usepackage[top=25mm,bottom=25mm,left=25mm,right=25mm]{geometry}
\usepackage{color}

\newtheorem{lemma}{Lemma}[section]

\newtheorem{theorem}{Theorem}[section]
\newtheorem{corollary}{Corollary}[section]

\newtheorem{example}{Example}[section]
\newtheorem{remark}{Remark}[section]
\newtheorem{assumption}{Assumption}[section]

\makeatletter
\def\section{\@startsection{section}{1}%
\z@{1\linespacing\@plus\linespacing}{1\linespacing}%
{\bf\centering}}
\def\subsection{\@startsection{subsection}{0}%
\z@{\linespacing\@plus\linespacing}{\linespacing}%
{\bf}}
\makeatother

\makeatletter
\@addtoreset{equation}{section}
\makeatother

\DeclareMathOperator{\sgn}{sgn}

\DeclareMathOperator{\Dom}{Dom}
\DeclareMathOperator{\Spec}{Spec}

\DeclareMathOperator{\var}{var}

\DeclareMathOperator{\bd}{b}

\DeclareMathOperator{\uppp}{U}
\DeclareMathOperator{\low}{L}

\providecommand{\pro}[1]{(#1_t)_{t \geq 0}}
\providecommand{\proo}[1]{(#1_t)_{t \in \R}}

\providecommand{\seq}[1]{(#1_n)_{n\in \mathbb{N}}}
\providecommand{\semi}[1]{\{#1_t: t \geq 0\}}

\newcommand{\ud}{\, \mathrm{d}}

\newcommand{\cL}{\mathcal{L}}
\newcommand{\cR}{\mathcal{R}}
\newcommand{\cK}{\mathcal{K}}

\newcommand{\cF}{\mathcal{F}}
\newcommand{\cB}{\mathcal{B}}

\newcommand{\cC}{\mathcal{C}}
\newcommand{\R}{\mathbb{R}}
\newcommand{\sS}{\mathbf{S}}
\newcommand{\1}{\mathbf{1}}

\newcommand{\pr}{\mathbb{P}}
\newcommand{\qr}{\mathbb{Q}}

\newcommand{\ex}{\mathbb{E}}
\newcommand{\Rd}{\mathbb{R}^d}
\newcommand{\N}{\mathbb{N}}

\begin{document}
\title%[LIL-type results for ground state-transformed jump processes]
{Typical long time behaviour of ground state-transformed jump processes}
\author{Kamil Kaleta and J\'ozsef L{\H o}rinczi}
\address{Kamil Kaleta, Faculty of Pure and Applied Mathematics
\\ Wyb. Wyspia{\'n}skiego 27, 50-370 Wroc{\l}aw, Poland}
\email{kamil.kaleta@pwr.edu.pl}

\address{J\'ozsef L\H orinczi,
Department of Mathematical Sciences, Loughborough University \\
Loughborough LE11 3TU, United Kingdom}
\email{J.Lorinczi@lboro.ac.uk}

\begin{abstract}
We consider a class of L\'evy-type processes derived via a Doob-transform from L\'evy processes
conditioned by a control function called potential. These processes have position-dependent and generally
unbounded components, with stationary distributions given by the ground states of the L\'evy generators
perturbed by the potential. We derive precise lower and upper envelopes for the almost sure long time
behaviour of these ground state-transformed L\'evy processes, characterized through escape rates and
integral tests. We also highlight the role of the parameters by specific examples.


\bigskip
\noindent
2010 {\it MS Classification}: Primary 47D08, 60G51; Secondary 47D03, 47G20
\end{abstract}

\thanks{Research of K. Kaleta was supported by the National Science Center, Poland, grant 2015/17/B/ST1/0123.}

\maketitle

\baselineskip 0.55 cm

%\newpage

\section{Introduction}

Given a random process $\pro X$, a fundamental question is what is its typical sample path behaviour on the long run.
This generally involves a statement of the form
\begin{equation}
\liminf_{t\to\infty} \frac{X_t}{\tau_1(t)} = C_1 \quad \mbox{and} \quad \limsup_{t\to\infty} \frac{X_t}{\tau_2(t)} = C_2,
\qquad \mathbb P-{\rm a.s.},
\label{LIL}
\end{equation}
where $\tau_1, \tau_2$ are positive functions on the positive semi-axis, $C_1, C_2$ are finite non-zero constants, and
$\mathbb P$ is the probability measure of the process. The functions $\tau_1, \tau_2$ provide lower and upper almost sure
envelopes, and thus give a characterization of the time-scale on which the process typically evolves in the long time limit.

In the present paper our aim is to consider this problem for a large class of L\'evy-type jump processes obtained from L\'evy
processes conditioned by Kato-class potentials, assuming that the so obtained processes have a stationary distribution. Such
processes arise from the Feynman-Kac representation of non-local Schr\"odinger operators of the form $H = -L +V$, where $L$
is the $L^2$-generator of a L\'evy process $\pro X$ on a suitable probability space, and $V$ is a multiplication operator
called potential. This representation reads
\begin{equation}
\label{FKf}
\left(e^{-tH}f\right)(x) = \ex^x[e^{-\int_0^t V(X_s)ds}f(X_t)], \quad f \in L^2(\R^d), \; x \in \R^d, \; t \geq 0,
\end{equation}
where the expectation is taken with respect to the probability measure of the process $\pro X$. Since the semigroup defined
by the right hand side is not measure preserving, using the ground state (i.e., eigenfunction at the bottom of the spectrum)
$\varphi_0$ of $H$ one can change the space $L^2(\R^d)$ to the weighted Hilbert space $L^2(\R^d,\varphi_0^2dx)$ on which the
correspondingly transformed semigroup becomes a Markov semigroup, and thus by a change of measure the right hand side in
\eqref{FKf} turns into an expectation with respect to a random process $\pro {\widetilde X}$ derived from $\pro X$ (for further
details see Section 2 below). We call such processes \emph{ground state-transformed (GST) processes}, which are thus a case
of Doob $h$-transformed processes, where the function $h$ is $\varphi_0$. The properties of this process will then be relevant
in a probabilistic study of the semigroup $\{e^{-tH}: t \geq 0\}$.

The ground state-transformed processes $\pro {\widetilde X}$ make a class of independent interest, even when the above
relevance is ignored. The generator of $\pro {\widetilde X}$ is
\begin{eqnarray}
(\widetilde H f)(x)
&=&
-\frac{1}{2} \sigma\nabla \cdot \sigma\nabla f(x) - \sigma\nabla\ln\varphi_0(x)\cdot
\sigma\nabla f(x) \nonumber \\
&& \qquad - \int_{0<|z|\le 1} \frac{\varphi_0(x+z)-\varphi_0(x)}{\varphi_0(x)} z\cdot \nabla f(x)\nu(z)dz \\
&& \qquad
-\int_{\R^d\setminus\{0\}}\big( f(x+z)-f(x)-z\cdot\nabla f(x)1_{\{{|z|\le 1}\}} \big)\frac{\varphi_0(x+z)}{\varphi_0(x)}
\nu(z)dz, \nonumber
\label{eq: generator}
\end{eqnarray}
where $\nu$ is the L\'evy intensity and $A = \sigma\sigma^T$ is the {diffusion matrix} of $\pro X$, and where we use the
notation $\sigma\nabla \cdot \sigma\nabla f(x) =\sum_{i,j=1}^d (\sigma\sigma^T)_{ij}\partial_{x_i} \partial_{x_j} f(x)$.
Under suitable conditions (see a discussion in \cite{LY}), the GST process satisfies a stochastic differential equation
with jumps of the form
\begin{eqnarray}
\label{sde}
\widetilde X_t
&=&
\widetilde X_0 + {\sigma} B_t + \int_0^t \sigma\nabla\ln \varphi_0(\widetilde X_s)\,ds
+ \int_0^t\int_{|z|\le 1}\frac{\varphi_0(\widetilde X_{s}+z)-\varphi_0(\widetilde X_{s})}{\varphi_0(\widetilde X_{s})} z\nu(z)dzds
 \nonumber \\
&& \quad
+\int_0^t\int_{|z|\le 1} \int_0^{\infty} z 1_{\left\{v\le \frac{\varphi_0(\widetilde X_{s-}+z)}{\varphi_0(\widetilde X_{s-})}\right\}}
\widetilde N(ds, dz, dv) \\
&& \quad
+\int_0^t \int_{|z|>1} \int_0^{\infty} z 1_{\left\{v\le \frac{\varphi_0(\widetilde X_{s-}+z)}{\varphi_0(\widetilde X_{s-})}\right\}}
N(ds, dz, dv) \nonumber,
\end{eqnarray}
where $\pro B$ is standard Brownian motion, $N$ is a Poisson random measure on $[0,\infty)\times\R^d\times[0,\infty)$ with intensity
$dt\nu(z)dz dv$, and $\widetilde N$ is the related compensated Poisson measure.

From the above two observations it is seen that the potential $V$ perturbing the L\'evy process enters the GST process via the
ground state $\varphi_0$ of the operator $H$, and in general gives rise to a position-dependent drift and a position-dependent bias in
the jump kernel,
i.e., a L\'evy-type process. Such processes are currently much researched on various levels of generality \cite{Sch,Sa,K}. Our class of
GST processes has the advantage that it has a definite structure while being a rich class, and its analysis depends on the properties
of a control function $V$ through $\varphi_0$. Also, from the expression in terms of the SDE above we note that GST processes have
unbounded coefficients, while most results on L\'evy-type processes have been established so far for bounded coefficients only (i.e.,
for cases when the symbol of the generator is uniformly bounded with respect to the position $x$ in space). Our goal in this paper is
to describe the profile function $\tau$ and the constant $C$ in function of the properties of $L$ and $V$.

The long term behaviour for the free processes, i.e., when the potential $V \equiv 0$, is described by classic results. When $\pro X$
is an $\R^d$-valued Brownian motion, Khinchin's law of iterated logarithm (LIL) says that \cite{Kha} the common envelope is described
by
$$
\tau(t) = \sqrt{2t \log\log t} \quad \mbox{and} \quad C=1.
$$
There is a rich literature on related results (e.g., the running maximum, characterization of limit points, local times, other
functionals of Brownian motion and random walk, large deviations, and similar problems on the typical short time behaviour, etc),
for some standard summaries see, e.g., \cite{CsR,R,Fr}.

This behaviour becomes very different in the case of heavy-tailed  purely jump processes.  Khinchin has also shown
\cite{Khb} (see important improvements in \cite{F,B}) that for non-Gaussian stable processes no similar LIL holds
in a very severe sense. If $\pro X$ is an isotropic $\alpha$-stable process with $0 < \alpha < 2$, then $C$ is either
zero or infinite for any positive increasing function $\tau$ on the positive semi-axis, according to whether
$\int_1^\infty \tau(t)^{-\alpha}dt$ is finite or infinite.  In contrast, for a real-valued L\'evy process $\pro X$
having a finite variance and zero mean Gnedenko proved \cite{G} that $\tau$ is the same as for Brownian motion and
$C = \sqrt{\var X_1}$. For processes which are spectrally one-sided or contain stable components etc see \cite{Z,B,PT},
and a standard modern summary is \cite{S}. For more recent results using Dirichlet forms see \cite{Sh,SW}. We note that
there is a large literature on short time LIL-type behaviour of L\'evy jump processes, however, the $t\downarrow 0$ limit
is beyond the scope of our paper.

Loosely speaking, the above results indicate that for a symmetric process the structure of the almost sure long time
profile $\tau$ is determined by the standard deviation and a small margin given by a slowly varying correction factor.
This margin can be further refined by integral tests.
Recall that $\tau$ is said to be in the upper resp. lower class at infinity with respect to $\pro X$ whenever
$\mathbb P(X_t < \tau(t) : \mbox{as $t \to \infty$})$ is 1 or 0. For Brownian motion, the so called Kolmogorov-Petrovsky
integral test says that if $g$ is a positive increasing function, then
$$
\mathbb P\left(|B_t| \leq \sqrt t g(t): \mbox{as $t \to \infty$}\right) = \; 0  \;\; \mbox{or} \;\; 1
$$
according to $\int_1^\infty \frac{g^d(t)}{t}e^{-\frac{g^2(t)}{2}}dt$ being finite or infinite. Also, the
Dvoretzky-Erd\H{o}s integral test says that if $h$ is a positive function, decreasing to zero, and $d \geq 3$,
then
$$
\mathbb P\left(|B_t| \geq \sqrt t h(t): \mbox{as $t \to \infty$} \right) = \; 0  \;\; \mbox{or} \;\; 1
$$
as $\int_1^\infty \frac{h^{d-2}(t)}{t}dt$ is finite or infinite. In particular, it follows that for some $n \in
\mathbb N$ and $d=3$,
$$
\tau(t) = \sqrt{2t\left(\log_2t + \frac{3}{2}\log_3t + \log_4t + ... + \log_{n-1}t +\left(1+\varepsilon\right) \log_n t\right)}
$$
where $\log_n$ means $n$-fold iterated logarithm, is in the upper or lower class at infinity, if $\varepsilon$ positive or
negative, respectively. For further integral tests related to Brownian motion and some jump processes we refer to \cite{Sh,ShN,K}.

The problem of long time behaviour has also been addressed for diffusions. In the works \cite{A1,A2} conditions have been
obtained for diffusions defined by stochastic differential equations such that the solutions continue to obey a LIL behaviour;
see also the classic paper by Motoo \cite{Mo}, and \cite{Ma} and the references therein. For GST processes obtained by
conditioning Brownian motion, Rosen and Simon \cite{RS} considered polynomial potentials increasing to infinity at infinity
and diffusions generated by the Schr\"odinger operator $-\frac{1}{2}\Delta + V$. They showed that if the degree of this polynomial
is $2m \geq 2$, and the coefficient of the leading term  is $a_{2m} > 0$, then the a.s. long-time profile of the GST process (called
by the authors $P(\phi)_1$-process)
is
$$
\tau(t) = (\log t)^\frac{1}{m+1} \quad \mbox{and} \quad C = 1/a_{2m}^{2(m+1)}.
$$
In \cite{BL}, more generally, Kato-class potentials $V$ were considered to study the support of Gibbs measures on Brownian paths,
see also \cite{LHB11}. Here it is shown that whenever the Schr\"odinger operator has a ground state $\varphi_0 \in L^1(\R^d)
\cap L^2(\R^d)$ and a spectral gap $\Lambda$, then the profile function of the so obtained two-sided diffusion is determined by
the condition
$$
\frac{e^{-\Lambda |t|}}{\varphi_0(X_t)} \to 0 \quad \mbox{as} \quad |t|\to\infty,
$$
from which explicit expressions can be derived for specific (classes of) examples. While this result has the advantage to
deal with a large class of potentials, it overestimates $\tau$ to large or small degrees dependent on $V$.

Long time behaviour for ground state-transformed jump L\'evy processes has been explored only for isotropic stable processes
so far, in the context of the fractional Laplacian $(-\Delta)^{\alpha/2}$, $0 < \alpha < 2$, see \cite{KL12}. In this paper
we go far beyond this class. Our main results are as follows. First we present an integral test for GST processes derived
from a general underlying L\'evy process conditioned by a general Kato-class potential (Theorem \ref{thm:int_test_1} and
Corollary \ref{cor:int_test_1} below). This will be achieved in terms of a functional directly featuring the ground state
(escape rates), to which we will be able to use the detailed information on their decay/concentration properties recently
obtained in \cite{KL15,KL17}. Next we restrict to a subclass of jump processes for which multiple large jumps are dominated
by single large jumps (which we call jump-paring L\'evy processes), and split the discussion to confining potentials ($V$
increasing to infinity at infinity) and decaying potentials ($V$ decreasing to zero at infinity), allowing us to get sharp
characterisations of the time evolution envelopes. For confining potentials we present an integral test in Theorem
\ref{thm:LIL_reg_var} and its implication on the long time behaviour in Corollary \ref{cor:LIL_reg_var}, and a similar pair
of results for decaying potentials in Theorem \ref{thm:LIL_reg_var_dec} and Corollary \ref{cor:LIL_reg_var_dec}. We refine
even further by assuming regular variation in Theorems \ref{thm:reg_prof}-\ref{thm:reg_prof_slow} in the case of
confining potentials, and slow variation in Corollaries \ref{cor:reg_prof_dec}-\ref{cor:reg_prof_slow_dec} in the case
of decaying potentials. We also prove the intuition that a faster decaying potential should imply tighter long time evolution
profiles (Theorem \ref{thm:monot}), and illustrate all these results by specific examples (Section 4.4) highlighting the
interplay of the L\'evy intensity and the potential in determining the growth of paths.

\section{The underlying and the ground state-transformed processes}
\subsection{Symmetric jump-paring L\'evy processes}

Let  $(X_t)_{t \geq 0}$ be a symmetric, $\R^d$-valued, $d \geq 1$, L\'evy process on a suitable probability space. We
use the notations $\pr^x$ and $\ex^x$ for the probability measure and expected value of the process starting in $x
\in \Rd$, respectively. The process $(X_t)_{t \geq 0}$ is determined by the characteristic function
$$
\ex^0 \left[e^{i \xi \cdot X_t}\right] = e^{-t \psi(\xi)}, \quad \xi \in \R^d, \ t>0,
$$
with exponent given by the L\'evy-Khintchin formula
\begin{align} \label{eq:LChE}
\psi(\xi) = A \xi \cdot \xi + \int_{\R^d} (1-\cos(\xi \cdot z)) \nu(dz).
\end{align}
Here $A$ is a symmetric non-negative definite $d \times d$ matrix, and $\nu$ is a symmetric L\'evy measure on $\R^d
\backslash \left\{0\right\}$, i.e., $\int_{\R^d} (1 \wedge |z|^2) \nu(dz) < \infty$ and $\nu(E)= \nu(-E)$, for all measurable
$E \subset \R^d \backslash \left\{0\right\}$, thus the L\'evy triplet of the process is $(0,A,\nu)$. We assume throughout that
the L\'evy measure is an infinite measure and it is absolutely continuous with respect to Lebesgue measure with density (L\'evy
intensity) $\nu(x) > 0$, i.e.,
\begin{align} \label{eq:nuinf}
\nu(\R^d \backslash \left\{0\right\})=\infty \quad \text{and} \quad \nu(dx)=\nu(x)dx.
\end{align}
When $A \equiv 0$ and $\nu \neq 0$, the L\'evy process $\pro X$ is a purely jump process, when $\nu \equiv 0$ and and $A \neq 0$,
it is purely continuous. Recall that $\pro X$ is a Markov process with respect to its natural filtration, satisfying the strong
Markov property and having c\`adl\`ag paths. Moreover, under \eqref{eq:nuinf} the process has the strong Feller property,
i.e., its transition semigroup satisfies $P_t(L^{\infty}(\R^d)) \subset C_{\rm b}(\R^d)$, for all $t>0$.
Equivalently, the one-dimensional distributions of $\pro X$ are absolutely continuous with respect to Lebesgue measure,
 i.e., there exist the transition probability densities $p(t,x,y) =p(t,y-x,0) =: p(t,y-x)$.
Its infinitesimal generator $L$ is uniquely determined by its Fourier symbol
\begin{align} \label{def:gen}
\widehat{L f}(\xi) = - \psi(\xi) \widehat{f}(\xi), \quad \xi \in \R^d, \; f \in \Dom L,
\end{align}
with domain $\Dom L=\left\{f \in L^2(\R^d): \psi \widehat f \in L^2(\R^d) \right\}$. It is a negative non-local self-adjoint
operator such that
$$
L f(x) = \sum_{i,j=1}^d a_{ij} \frac{\partial^2 f}{\partial x_j \partial x_i} (x)
+ \int \left(f(x+z)-f(x)-\1_{B(0,1)}(z)z \cdot \nabla f(x)\right) \,\nu(z) dz, \quad x \in \R^d,
$$
for $f \in C_0^{\infty}(\R^d)$. For more details on L\'evy processes we refer to \cite{S, B}.

%We will work in Sections 2 and 3 below with general underlying symmetric L\'evy processes determined by the L\'evy-Khintchin
%exponent \eqref{eq:LChE}. However, our main and sharpest results in Section 4 are obtained for a restricted class of processes
%defined as follows. In Section 5 we also study the purely diffusive (in fact, Brownian) case when $\nu \equiv 0$.

In what follows we will also consider a more restricted class of symmetric L\'evy processes defined by a condition on the large
jumps, introduced in \cite{KL12}. Recall the following standard notations. For given functions $f, g$ the notation $f \asymp C g$
means that $C^{-1} g \leq f \leq Cg$ with a constant $C$, and $f \asymp g$ means that there is a constant $C$ such that this
relation holds. Also, we write $f \approx g$ when $\lim_{r \to \infty} f(r)/g(r) = 1$. The constants will be assumed to be
dependent on the dimension $d$ by default, while dependence of $C$ on the process $\pro X$ will indicated by $C(X)$.

\begin{assumption}
\label{ass:assnu}
The following conditions hold:
\begin{enumerate}
\item
There exist a non-increasing function $f:(0,\infty) \to (0,\infty)$ and constants $C_1, C_2, C_3 > 0$
such that
\begin{align} \label{eq:ass1}
C_1 f(|x|) \leq \nu(x) \leq C_2 f(|x|), \quad x \neq 0,
\end{align}
and
\begin{align} \label{eq:ass2}
\int_{|y|>1/2, \,  |x-y|>1/2} f(|x-y|) f(|y|) dy \leq C_3 \: f(|x|), \quad |x|\geq 1.
\end{align}
\item
There exist $t_{\rm b}>0$ and
$C_4=C_4(X,t_{\rm b})$ such that $0<p(t_{\rm b},x) \leq C_4$, for all $x \in \R^d$.
\item
For all $0<p<q<R \leq 1$ we have $\sup_{x \in B(0,p)} \sup_{y \in B_q(0)^c} G_{B_R(0)}(x,y) < \infty$,
where $G_{B_R(0)}(x,y) = \int_0^{\infty} p_{B_R(0)}(t,x,y) dt$ denotes the Green function of the process
$\pro X$ in the ball ${B_R(0)}$.
\end{enumerate}
\end{assumption}
\noindent
We refer to the class of L\'evy processes satisfying Assumption \ref{ass:assnu} as \emph{symmetric jump-paring L\'evy
processes}, and to condition \eqref{eq:ass2} as the \emph{jump-paring property}. It means that double (and by iteration,
all multiple) large jumps are stochastically dominated by single large jumps. This condition has been introduced in
\cite{KL15}, for its further uses see also \cite{KS14,KL17}.

\begin{example}
{\rm
The jump-paring class has a non-trivial overlap with subordinate Brownian motions in the sense that neither contains
the other class. Some landmark examples include:
\begin{itemize}
\item[(1)]
isotropic $\alpha$-stable processes, generated by $L = (-\Delta)^{\alpha/2}$, $0 < \alpha < 2$
\item[(2)]
isotropic relativistic $\alpha$-stable processes, generated by $L = (-\Delta + m^{2/\alpha})^{\alpha/2} - m$,
$0 < \alpha < 2$, $m > 0$
\item[(3)]
isotropic geometric $\alpha$-stable processes, generated by $L = \log(1 + (-\Delta)^{\alpha/2})$, $0 < \alpha < 2$
\item[(4)]
jump-diffusion processes obtained as the sum of a mutually independent Brownian motion and an isotropic $\alpha$-stable
process, generated by $L = -a\Delta + b(-\Delta)^{\alpha/2}$, $0 < \alpha < 2$, $a, b > 0$.
\end{itemize}
In contrast, the variance gamma process corresponding to an $\alpha = 2$ geometric stable process does not belong to
the jump-paring class. For a more detailed discussion of special cases and examples we refer to \cite{KL15}.
}
\end{example}

The restricted class of processes given by Assumption \ref{ass:assnu} will be used only in Section 4 below. For the
remainder of this section $\pro X$ denotes a general symmetric L\'evy process corresponding to the L\'evy-Khintchin
exponent (\ref{eq:LChE}).

\subsection{Ground state-transformed processes}

\subsubsection{\rm \textbf{Potentials and Feynman-Kac semigroup.}}
Below we will consider L\'evy processes conditioned by appropriate potentials.
%A potential is a Borel function $V: \R^d \to \R$, and we will assume that they satisfy the following conditions.
%\begin{definition}
Recall that a Borel measurable function $V: \R^d \to \R$ is an \emph{$X$-Kato class potential} whenever for its
positive and negative parts
\begin{equation}
\label{ass:pots}
V_{-} \in \cK^X \quad \mbox{and} \quad V_+  1_C \in \cK^X \; \mbox{for every compact subset $C \subset \R^d$},
\end{equation}
holds, where $h \in \cK^X$ means that
\begin{align}
\label{eq:Katoclass}
\lim_{t \downarrow 0} \sup_{x \in \R^d} \ex^x \left[\int_0^t |h(X_s)| ds\right] = 0.
\end{align}
\vspace{0.1cm}
%\item
%is a \emph{pinning potential}, i.e., $V(x) \to \infty$ as $|x| \to \infty$.
%\vspace{0.1cm}
%\item
%a \emph{decaying potential}, i.e., $V(x) \to 0$ as $|x| \to \infty$.
%\item
%$V$ is a \emph{perturbation}, i.e., the operator $H := H_0 + V$ is a self-adjoint perturbation in Kato sense
%of $H_0$, where $H_0$ is the infinitesimal generator of the L\'evy process $\pro X$.
%\end{enumerate}
%\end{definition}

By an extension of Khasminskii's lemma \cite[Lem. 3.37]{LHB11} to $X$-Kato potentials, it follows that the random
variables $\int_0^t V(X_s) ds$ are exponentially integrable for all $t \geq 0$, and thus we can define the Feynman-Kac
semigroup
\begin{equation}
T_t f(x) = \ex^x\left[e^{-\int_0^t V(X_s) ds} f(X_t)\right], \quad f \in L^2(\R^d), \ t \geq 0, \ x \in \R^d.
\label{FK}
\end{equation}
Using the Markov property and stochastic continuity of the process $\pro X$ it can be shown that $\{T_t: t\geq 0\}$
is a strongly continuous one-parameter semigroup of symmetric operators on $L^2(\R^d)$. Moreover, by the Hille-Yoshida
theorem there exists a self-adjoint operator $H$ bounded from below such that $e^{-tH} = T_t$. The generator can be
identified as the non-local Schr\"odinger operator $H = - L + V$ defined as a form sum, where $L$ is the infinitesimal
generator of the L\'evy process $\pro X$.

%Denote
%$$
%\lambda_0 := \inf\Spec H.
%$$
The following will be a basic standing assumption for the whole paper.

\begin{assumption}
\label{ass:gsex}
The potential $V$ is in $X$-Kato class, chosen such that $\lambda_0 := \inf\Spec H \in \R$ is an isolated eigenvalue of
$H$.
\end{assumption}

\noindent
We denote the corresponding eigenfunction (called \emph{ground state}) by $\varphi_0$, i.e.,
$$
H \varphi_0 = \lambda_0 \varphi_0, \;\; \varphi_0 \not \equiv 0, \;\; \varphi_0 \in \Dom H \subset L^2(\R^d)
$$
holds. By standard arguments \cite[Th.XIII.43]{ReS}, \cite[Sect. 3.4.3]{LHB11} it follows that $\varphi_0$ is
unique and has a strictly positive version, which we will use throughout below.

Both from the perspective of existence of a ground state and for the purposes of the discussion below, it is
useful to single out two large classes of potentials.
\begin{example}[\textbf{Confining potentials}]
\label{confipot}
{\rm
A potential $V$ is \emph{confining} if $V(x) \to \infty$ as $|x| \to \infty$. In this case $\Spec H$ is purely discrete,
and a (unique) ground state $\varphi_0$ exists. Some examples include:
\begin{itemize}
\item[(1)]
\emph{Harmonic and anharmonic oscillators:} Let $V(x)= |x|^{2n}$, $n \in \N$. The case $n=1$ describes the potential
of the harmonic oscillator, and $n \geq 2$ give anharmonic oscillators.

\vspace{0.1cm}
\item[(2)]
\emph{Double and multiple well potentials:} The potential $V(x) = |x|^4 - b|x|^2$, $b > 0$, is a symmetric double well
potential. Multiple well potentials can be obtained by higher order polynomials.
\end{itemize}
}
\end{example}

\begin{example}[\textbf{Decaying potentials}]
\label{decaypot}
{\rm
A potential $V$ is \emph{decaying} if $V(x) \to 0$ as $|x| \to \infty$. In this case $\Spec H$ contains the essential spectrum
$\Spec_{\rm ess} H = \Spec_{\rm ess} L = [0,\infty)$, and whether it also contains a non-empty discrete component depends
on further details of $V$. Some decaying $X$-Kato class potentials of special interest in mathematical physics are:
\begin{itemize}
\item[(1)]
\emph{Potential wells:} Let $V(x)= - v(x)$ with a compactly supported, non-negative bounded Borel function $v \not
\equiv 0$. Specifically, we can choose $V(x)=-a \1_{B(0,1)}(bx)$, for $a, b>0$.

\vspace{0.1cm}
\item[(2)]
\emph{Coulomb-type potentials:} Let $f$ in Assumption \ref{ass:assnu} be such that $f(r)=r^{-d-\alpha}$, $r \in (0,1]$,
for some $\alpha \in (0,2)$, and let $V(x)=-(a_1|x|^{-\beta_1} \wedge a_2|x|^{-\beta_2})$, with $\beta_1 \in
(0,\alpha \wedge d]$, $\beta_2 \in [\beta_1, \infty)$ and $a_1, a_2 >0$.

\vspace{0.1cm}
\item[(3)]
\emph{Yukawa-type potentials:}
Let $f$ in Assumption \ref{ass:assnu} be as in (2) above and $V(x)=-(a_1|x|^{-\beta_1} \wedge a_2|x|^{-\beta_2}e^{-b|x|})$,
with $\beta_1 \in (0,\alpha \wedge d]$, $\beta_2 \in [\beta_1, \infty)$ and $a_1, a_2, b >0$.

\vspace{0.1cm}
\item[(4)]
\emph{P\"oschl-Teller potential:} $V(x)= -a/\cosh^2 (b|x|)$, with $a, b > 0$.

\vspace{0.1cm}
\item[(5)]
\emph{Morse potential:} $V(x)= a((1-e^{-b(|x|-r_0)})^2-1)$, with $a, b, r_0 > 0$.
\end{itemize}
}
\end{example}

\subsubsection{\rm \textbf{Ground state-transformed process.}}
By using $\varphi_0$, we define the ground state transform as the unitary map
$$
U: L^2(\R^d,\varphi_0^2dx) \to L^2(\R^d,dx), \quad f \mapsto \varphi_0 f.
$$
Also, we define the intrinsic Feynman-Kac semigroup
\begin{equation}
\label{IFK}
\widetilde{T}_t f(x) = \frac{e^{\lambda_0 t}}{\varphi_0(x)} T_t(\varphi_0 f)(x)
\end{equation}
associated with $\semi T$. Using the integral kernel $u(t,x,y)$ of $T_t$ we have then that
$\widetilde{T}_t f(x) = \int_{\R^d} \tilde u(t,x,y)f(y) \varphi_0^2(y)dy$ with the integral
kernel given by
\begin{equation}
\widetilde{u}(t,x,y) = \frac{e^{\lambda_0 t} u(t,x,y)} {\varphi_0(x)\varphi_0(y)},
\label{kerIFK}
\end{equation}
and infinitesimal generator $\tilde H = U^{-1}(H-\lambda_0)U$, with domain
$$
\Dom \widetilde H = \{f \in L^2(\R^d, \varphi_0^2dx): Uf \in \Dom H \}.
$$
A calculation then shows that $\widetilde H$ is given by the expression \eqref{eq: generator}. Furthermore, the
operators $\widetilde T_t$ are contractions and we have $\widetilde T_t \1_{\R^d} = \1_{\R^d}$ for all $ t \geq 0$,
thus $\semi {\widetilde T}$ is a Markov semigroup on $L^2(\R^d, \varphi_0^2dx)$.

The self-adjoint operator $\widetilde H$ generates a stationary Markov process, which we call a \emph{ground state-transformed
(GST) process}. (In the terminology of \cite{RS} it is called a $P(\phi)_1$-process associated with potential $V$.)
To define GST processes, we need two-sided underlying processes. Denote by $\Omega_{\rm r}$ the space of right continuous
functions from $[0,\infty)$ to $\R^d$ with left limits (i.e., c\`adl\`ag functions), and by $\Omega_{\rm l}$ the space of
left continuous functions from $[0,\infty)$ to $\R^d$ with right limits (i.e., c\`agl\`ad functions). Denote the corresponding
Borel $\sigma$-fields by $\mathcal B(\Omega_{\rm r})$ and $\mathcal B(\Omega_{\rm l})$, respectively. Let $\pro {X^{\rm r}}$ be
a L\'evy process on the space $(\Omega_{\rm r},\mathcal B(\Omega_{\rm r}),\pr^x_{\rm r})$, where $X_t^{\rm r}(\omega) = \omega(t)$
is the coordinate process on $\Omega_{\rm r}$, and let $\pro {X^{\rm l}}$ be a L\'evy process on the space $(\Omega_{\rm l},
\mathcal B(\Omega_{\rm l}),\pr^x_{\rm l})$, where $X_t^{\rm l}(\varpi) = \varpi(t)$ is the coordinate process on $\Omega_{\rm l}$.
Consider the product probability space $(\Omega_{\rm r} \times \Omega_{\rm l}, \mathcal B(\Omega_{\rm r}) \times
\mathcal B(\Omega_{\rm l}), \pr^x_{\rm r} \otimes \pr^x_{\rm l})$, and for every $\hat \omega = (\omega, \varpi) \in
\Omega_{\rm r} \times \Omega_{\rm l}$ define
\begin{align}
 \hat X_t(\hat\omega)
 = \left\{
  \begin{array}{ccc}
    \omega(t) & \;\; \text{\rm if} \; \; t \geq 0,\\
    \varpi(-t)  & \;\; \text{\rm if} \; \; t < 0.
  \end{array}\right.
 \end{align}
Then $t \mapsto  \hat X_t(\cdot)$ is a c\`adl\`ag function for all $t\in \R$. Denote by $\Omega$ the space of c\`adl\`ag functions
$\R \to \R^d$, with Borel $\sigma$-field by $\mathcal B(\Omega)$. Consider the image measure $\qr^x = (\pr^x_{\rm r} \otimes
\pr^x_{\rm l}) \circ \hat X^{-1}_\cdot$. Then the coordinate process $\proo Y$ on $(\Omega, \mathcal B(\Omega),\qr^x)$ is a L\'evy
process such that $\qr^x(Y_0=x) =1$, the increments $(Y_{t_i} - Y_{t_{i-1}})_{1 \leq i \leq n}$ are independent and stationary for
every $0 = t_0 < ... < t_n$, $n \in \N$, the increments $(Y_{-t_{i-1}}-Y_{-t_i})_{1 \leq i \leq n}$ are independent and stationary
for every $0 = - t_0 > ... > - t_n$, $n \in \N$, and the function $\R \ni t \mapsto Y_t(\cdot) \in \R^d$ is $\qr^x$-a.s. c\`adl\`ag.

Using two-sided c\`adl\`ag path space, we can now define GST processes. The following result gives
the existence and fundamental properties of GST processes for general underlying L\'evy processes
and general Kato-class potentials. A first variant for jump processes has been obtained in \cite[Th. 5.1]{KL12} for GST processes
derived from isotropic stable processes, but the argument is generic and it applies directly to the present settings,
see for further details \cite[Th. 2.1]{LY}.
For an initial variant of the concept defined for an underlying Brownian motion and allowing simplifications
due to path continuity we refer to \cite{S2,BL}. For infinite dimensional GST processes we refer to \cite{G1,G2};
see also a detailed discussion in \cite{LHB11}.
\begin{theorem}[\textbf{Ground state-transformed process}]
\label{th:exphi1}
Let $V$ be an $X$-Kato class potential and $\semi{\widetilde T}$ be the corresponding intrinsic Feynman-Kac semigroup. For all $x \in
\R^d$ there exists a probability measure $\widetilde \pr^x$ on $(\Omega, \mathcal B(\Omega))$  and a random process $\proo {\widetilde X}$
satisfying the following properties:
 \begin{itemize}
\item[(1)]
Let $-\infty < t_0 \leq t_1 \leq ... \leq t_n < \infty$ be an arbitrary division of the real line, for any
$n \in \N$. The initial distribution of the process is
$$
\widetilde \pr^x(\widetilde{X}_0 = x) = 1,
$$
and the finite dimensional distributions of  $\widetilde \pr^x$ with respect to the stationary distribution
$\varphi_0^2 dx$ are given by
\begin{align}
\label{eq:fddist}
\int_{\R^d} \ex_{\widetilde \pr^x}\Big[\prod_{j=0}^n f_j(\widetilde{X}_{t_j})\Big] \varphi^2_0(x) dx =
\left(f_0,\: \widetilde{T}_{t_1-t_0}\: f_1 ...\: \widetilde{T}_{t_n-t_{n-1}}\: f_n \right)_
{L^2(\R^d, \varphi_0^2 dx)}
\end{align}
for all $f_0,f_n \in L^2(\R^d, \varphi_0^2dx)$, $f_j \in L^{\infty}(\R^d)$, $j=1,..., n-1$.
\item[(2)]
The finite dimensional distributions are time-shift invariant, i.e.,
$$
\int_{\R^d} \ex_{\widetilde \pr^x}\Big[\prod_{j=0}^n f_j(\widetilde{X}_{t_j})\Big] \varphi^2_0(x) dx =
\int_{\R^d} \ex_{\widetilde \pr^x}\Big[\prod_{j=0}^n f_j(\widetilde{X}_{t_j+s})\Big] \varphi^2_0(x) dx,
\quad s \in \R, \, n \in \N.
$$
\item[(3)]
$(\widetilde{X}_t)_{t \geq 0}$ and $(\widetilde{X}_t)_{t \leq 0}$ are independent, and $\widetilde X_{-t}
\stackrel{\rm d}{=} \widetilde X_t$, for all $t \in \R$.
\item[(4)]
With the filtrations $(\cF_t^{+})_{t \geq 0} = \sigma\{\widetilde{X}_s: 0 \leq s \leq t\}$  and $(\cF_t^{-})_{t \leq 0}
= \sigma\{\widetilde{X}_s: t \leq s \leq 0\}$, the random process $(\widetilde{X}_t)_{t \geq 0}$ is a Markov process
with respect to $\left(\cF_t^{+}\right)_{t \geq 0}$, and $(\widetilde{X}_t)_{t \leq 0}$ is a Markov process with
respect to $\left(\cF_t^{-}\right)_{t \leq 0}$.
\end{itemize}
Furthermore, we have for all $f,g\in L^{2}(\R^d,\varphi_0^2dx)$ the change-of-measure formula
\begin{equation}
(f, \widetilde T_t g)_{L^{2}(\R^d,\varphi_0^2dx)}= (f\varphi_0, e^{-t(H-\lambda_0 )}g\varphi_0)_{L^2(\R^d,dx)}=
\int_{\R^d}\ex_{\widetilde \pr^x}[f(\widetilde X_0) g(\widetilde X_t)]\varphi_0^2{(x)}dx, \quad t \geq 0.
\label{gibbs}
\end{equation}
In particular, we have the path measure
\begin{equation}
\label{pphi}
\widetilde \pr (A)= \int_{\R^d} \ex_{\widetilde \pr^x}\left[1_A \right] \varphi^2_0(x) dx, \quad
A \in \mathcal B(\Omega).
\end{equation}
\end{theorem}
\noindent

\newpage
\begin{remark}
\hspace{100cm}
{\rm
\begin{trivlist}
\item[\ (1)]
As shown in \cite[Th. 3.1]{LY}, under the condition that $x \mapsto \nabla \log \varphi_0(x)$ is locally bounded,
a GST process $\pro {\widetilde X}$ satisfies the SDE given in \eqref{sde}. We will discuss some specific cases
below.

\item[\ (2)]
The probability measure $\widetilde \pr^x$ can be seen as a Gibbs measure on the space of two-sided c\`adl\`ag
paths. Consider a regular version of the regular conditional probability measure $\widetilde \pr^{x,s}_{y,t}(\:\cdot\:)
= \widetilde \pr(\: \cdot \: | \widetilde X_s = x, \widetilde X_t = y)$, $x, y \in \R^d$, $s < t \in \R$, and the (not
normalized) measure on $(\Omega|_{\left[s,t\right]}, \mathcal B(\Omega|_{\left[s,t\right]}))$ corresponding to the L\'evy
bridge process $(X_r)_{s\leq r\leq t}$ given by
$$
b^{x,y}_{[s,t]}(\:\cdot\:) = p(t-s,y-x) \pr^{x,s}_{y,t}(\:\cdot\:).
$$
Then (\ref{pphi}) can be equivalently written as
$$
\widetilde \pr (A) = \int_{\R^d}dx  \varphi_0(x) \int_{\R^d} dy  \varphi_0(y)
\int_{\Omega} e^{-\int_s^t (V(X_r(\omega))-\lambda_0) dr} 1_A db_{[s,t]}^{x,y}(\omega)
$$
for all $A \in \mathcal B(\Omega|_{\left[s,t\right]})$ and all $s < t \in \R$. It can be shown that the family of
conditional probabilities indexed by the family of intervals $[s,t]$ and given by the last integral above satisfies
the Dobrushin-Lanford-Ruelle consistency relations, and thus $\widetilde \pr$ is a Gibbs measure on $(\Omega, \mathcal
B(\Omega))$ with respect to the potential $V$. The details are left to the interested reader; for a discussion of Gibbs
measures relative to stable processes see \cite[Sect. 5.3]{KL12}, which can be extended through similar steps. Our results
below on the almost sure long time behaviour of GST processes will then also characterize the supports of these Gibbs
measures.

\vspace{0.1cm}
\item[\ (3)]
When $V$ is a confining potential, the process $(\widetilde X_t, \widetilde \pr^x)_{t \geq 0, \, x \in\R^d}$
is typically $\varphi_0^2dx$-recurrent, i.e., for every $x \in \R^d$ and Borel set $A \subset \R^d$ such that
$\int_A \varphi_0^2(y)dy >0$ (or, equivalently, with positive Lebesgue measure) we have $\int_0^{\infty} \widetilde
\pr^x(\widetilde X_t \in A) dt = \infty$. If there exists $g:[0,\infty) \to [0,\infty)$, $g(r) \nearrow \infty$ as
$r \to \infty$, $g(r+1) \asymp g(r)$, $r \geq 1$, such that $V(x) \asymp g(|x|)$, then it follows from the estimates
of the kernel $u(t,x,y)$ \cite[Cor. 4.7]{KSch18} that there exists $t_0>0$ such that for every $t \geq t_0$, $x \in \R^d$
and $A \subset \R^d$ as above it holds that $\widetilde \pr^x(\widetilde X_t \in A) \geq c$, with a constant $c=c(x,A) >0$.

\end{trivlist}
}
\end{remark}

For cases when $\varphi_0$ is explicitly known, we can construct specific GST processes which give further insight.
\begin{example}[\textbf{GST Brownian motion}]\label{GST-BM}
{\rm
First consider the underlying L\'evy process $\pro X$ to be a standard Brownian motion. Though we discuss the
one-dimensional cases only, the first two examples below can be extended to arbitrary finite dimension.
\begin{itemize}
\item[(1)]
\emph{Ornstein-Uhlenbeck process:} Let $H = -\frac{1}{2}\frac{d^2}{dx^2} + V$, with confining potential $V(x) =
\frac{\gamma^2}{2}x^2 - \frac{\gamma}{2}$, $\gamma > 0$. A calculation gives
$$
\varphi_0(x) = \sqrt[4]{\frac{\gamma}{\pi}} \ e^{-\frac{\gamma x^2}{2}} \quad \mbox{and} \quad
\widetilde H = -\frac{1}{2}\frac{d^2}{dx^2} + \gamma x \frac{d}{dx}.
$$
Hence we have the GST process satisfying the SDE
$$
dX_t = -\gamma X_t dt + dB_t, \quad X_0 = a,
$$
i.e., the Ornstein-Uhlenbeck process $X_t = ae^{-\gamma t} + \int_0^t e^{-\gamma(t-s)} dB_s$. The role of the potential
appears in a strong killing, which makes the process favour the region around the origin, and spend proportionally less
time further away.

\vspace{0.1cm}
\item[(2)]
\emph{Brownian motion in a finite  potential well:} Let $H = -\frac{1}{2}\frac{d^2}{dx^2} + V$, with compactly
supported potential $V(x) = - v 1_{\{|x| \leq a\}}$, $a, v > 0$. We have
$$
\varphi_0(x) = A_0 e^{-\sqrt{2|\lambda_0|} |x|}1_{\{|x| > a\}} +  B_0 \cos(\sqrt{2(v-|\lambda_0|)}x)1_{\{|x| \leq a\}},
$$
where $A_0, B_0$ can be determined by the normalization condition $\|\varphi_0\|_2=1$, and the ground state eigenvalue
is the smallest solution $\lambda = \lambda_0$ of the transcendental equation $\tan(a\sqrt{2(v-|\lambda|)}) =
\sqrt{\frac{\lambda}{v-\lambda}}$. Using that $a\sqrt{2(v-|\lambda_0|)} < \frac{\pi}{2}$, we obtain that the GST
process satisfies the equation $dX_t = b(X_t)dt + dB_t$, with drift term
$$
\qquad b(X_t)
= - \sqrt{2|\lambda_0|} \, \sgn(X_t)1_{\{|x| > a\}} -
\sqrt{2(v-|\lambda_0|)} \tan \left(\sqrt{2(v-|\lambda_0|)}X_t\right)1_{\{|x| \leq a\}}.
$$
From the above it can be seen that as soon as a path exits the potential well, it is pulled back by the drift at a
constant speed $\sqrt{2|\lambda_0|}$, which will act as a basic mechanism preventing explosion.

\vspace{0.1cm}
\item[(3)]
\emph{Diffusions with Pearson distributions:} It is a yet little explored though notable fact that the six classes of
Pearson distribution correspond to classical Schr\"odinger operators with P\"oschl-Teller, Morse etc potentials given
in Example \ref{decaypot} above. For further details see \cite[Table 10]{ALS}.
\end{itemize}
}
\end{example}

\begin{example}[\textbf{GST Cauchy process}]
{\rm
Let $H = (-\frac{1}{2}\frac{d^2}{dx^2})^{1/2} + V$. Using the results in \cite{LM}, in which explicit solutions have been
obtained for the harmonic potential $V(x) = x^2$, and in \cite{DL} for the anharmonic potential $V(x) = x^4$, one can
construct related GST processes for the one-dimensional 1-stable (i.e., Cauchy) process generated by the square root of the
one-dimensional negative Laplacian. In this case we have
\begin{align*}
\widetilde H f(x) & = -c_d \int_{\R^d\setminus\{0\}}\big( f(x+z)-f(x)-z\cdot\nabla f(x)1_{\{{|z|\le 1}\}}
\big)\frac{\varphi_0(x+z)}{\varphi_0(x)} |z|^{-d-1}dz \\
& \ \ \ \ \ \ \ \ \ - c_d \int_{0<|z|\le 1} \frac{\varphi_0(x+z)-\varphi_0(x)}{\varphi_0(x)} z\cdot \nabla f(x)|z|^{-d-1}dz,
\end{align*}
from which a specific case of \eqref{sde} can be obtained.
}
\end{example}

\section{Integral tests and long time behaviour for general jump GST-processes}\label{sec:general}

%\subsection{An extension of the Borel-Cantelli lemma}
\subsection{Technical lemmas}
\noindent
In this section we consider general underlying L\'evy processes defined by the exponent \eqref{eq:LChE}, i.e.,
do not make the restriction to the jump-paring class given by Assumption \ref{ass:assnu}.

We start by an extension of the Borel-Cantelli lemma, which also extends a result in \cite{RS}. The first statement
is a direct consequence of the classical Borel-Cantelli lemma, while the second uses the concept of $h$-mixing (for
a definition we refer to \cite[eq.(6a)]{RS}).

\begin{lemma} \label{lem:fluct}
Suppose a function $\tau :\N \to (0,\infty)$ is given.
\begin{itemize}
\item[(1)]
If  \ $\sum_{n=1}^\infty \widetilde \pr \big(|\widetilde X_n| \geq \tau(n) \big) < \infty$, then $|\widetilde X_n| < \tau(n)$
for almost every $n \in \N,$ $\widetilde \pr$-a.s.
\item[(2)]
If \ $\sum_{n=1}^\infty \widetilde \pr \big(|\widetilde X_n| > \tau(n)\big) = \infty$, then $|\widetilde X_n| \geq \tau(n)$
for infinitely many $n \in \N$, $\widetilde \pr$-a.s.
\end{itemize}
\end{lemma}

\begin{proof}
When $\sum_{n=1}^\infty \widetilde \pr \big(|\widetilde X_n| > \tau(n)\big) < \infty$, the Borel-Cantelli lemma gives
$|\widetilde X_n| \leq \tau(n)$ for almost all $n \in \N$, $\widetilde \pr$-a.s., and thus (1) holds.
To obtain (2), let $\cF_n = \sigma \{\widetilde X_t\, : \, n \leq t \leq n+1\}$, for $n \in \N$. By using
that $\lambda_1 - \lambda_0 >0$ and the same argument as in \cite[Th.3]{RS} we find that the family of sigma-fields
$\seq\cF$ is $h$-mixing with the function $h(n):=e^{-(\lambda_1-\lambda_0)n}$, $n \in \N$. Therefore, if $\sum_{n=1}^\infty
\widetilde \pr\big(|\widetilde X_n| > \tau(n) \big) = \infty$, then by \cite[Th.2(8b)]{RS} it follows that $|\widetilde X_n| \geq
\tau(n)$ for infinitely many $n \in \N$, $\widetilde \pr$-a.s., and (2) holds.
\end{proof}

Next we establish an estimate needed to control the series appearing in the previous lemma, which will play an
essential role below.
\begin{lemma} \label{lem:series_int}
Let $(X_t)_{t \geq 0}$ be a L\'evy process determined by \eqref{eq:LChE}, $V$ a potential satisfying Assumption
\ref{ass:gsex}, and $\pro {\widetilde X}$ the corresponding GST-process with probability measure $\widetilde \pr$. Then for
every non-decreasing function $\tau:[0, \infty) \to (0,\infty)$ we have
$$
\sum_{n=1}^{\infty} \widetilde \pr \left(|\widetilde X_n| \geq \tau(n) \right) < \infty \ \ \Longleftrightarrow\ \
\int_1^{\infty} dr \int_{|x| \geq \tau(r)}\varphi_0^2(x) dx < \infty.
$$
\end{lemma}

\begin{proof}
First notice that by monotonicity of $\tau$ we have
$$
\int_1^{\infty} \widetilde \pr \left(|\widetilde X_t| \geq \tau(r) \right) dr < \infty \ \Longleftrightarrow \ \sum_{n=1}^{\infty}
\widetilde \pr \left(|\widetilde X_t| \geq \tau(n) \right) < \infty, \quad t \geq 0.
$$
Also, recall that there exist constants $c_1, c_2 >0$ such that for every $r > 0$ and $t \geq 0$
$$
c_1 \int_{|x| \geq r} \varphi_0^2(x) dx \leq \widetilde \pr \left(|\widetilde X_t| \geq r \right) \leq c_2 \int_{|x| \geq r}
\varphi_0^2(x) dx
$$
and thus
$$
c_1 \int_1^{\infty} \int_{|x|>\tau(r)} \varphi_0^2(x) dx dr \leq \int_1^{\infty}  \widetilde \pr \left(|\widetilde X_t|
\geq \tau(r) \right) dr \leq c_2 \int_1^{\infty}  \int_{|x|>\tau(r)} \varphi_0^2(x) dx dr,
$$
which completes the proof.
\end{proof}

\subsection{General integral test and almost sure long-time behaviour}
\noindent
First we present an integral test to GST-processes obtained for general L\'evy processes in the above framework. For $c>0$
and a non-decreasing function $\tau:[0,\infty) \to (0, \infty)$ define
$$
I_{\varphi_0}(c,\tau) := \int_1^{\infty} dr \int_{|x|\geq \tau(r)} \varphi_0^2(c x) dx =
\int_1^{\infty} \tau^d(r) \int_{|x|\geq 1} \varphi_0^2(c \tau(r) x) dx dr
$$
and
$$
c_{\varphi_0}(\tau):= \inf \left\{c>0: I_{\varphi_0}(c,\tau) < \infty\right\}.
$$
Clearly, in general, $c_{\varphi_0}(\tau) \in [0,\infty]$, and the integral $I_{\varphi_0}(c,\tau)$ can be seen as an
escape rate for given $\tau$.

\smallskip

The following 0-1 criterion holds.

\begin{theorem}[\textbf{Integral test: general underlying process}]
\label{thm:int_test_1}
Let $(X_t)_{t \geq 0}$ be a L\'evy process determined by \eqref{eq:LChE}, $V$ a potential satisfying Assumption
\ref{ass:gsex}, and $\pro {\widetilde X}$ the corresponding GST-process with probability measure $\widetilde \pr$. Then for
every non-decreasing function $\tau : [0,\infty) \to (0,\infty)$ we have
 \begin{align} \label{eq:test_1}
\widetilde \pr \left(|\widetilde X_n| \geq \tau(n) \  \text{\rm for infinitely many} \ n \in \N \right) = \left\{
  \begin{array}{ccc}
    0 & \quad \text{\rm if} \quad I_{\varphi_0}(1,\tau) < \infty,\\
    1 & \quad \text{\rm if} \quad I_{\varphi_0}(1,\tau) = \infty.
  \end{array}\right.
 \end{align}
\end{theorem}
\begin{proof}
The equalities in (\ref{eq:test_1}) follow directly from Lemmas \ref{lem:series_int} and \ref{lem:fluct}.
\end{proof}

\begin{corollary}[\textbf{Long-time behaviour: general underlying process}]
\label{cor:int_test_1}
Under the conditions of Theorem \ref{thm:int_test_1} we have that
\begin{align}
\label{eq:LIL1}
\limsup_{n \to \infty} \frac{|\widetilde X_n|}{\tau(n)} = c_{\varphi_0}(\tau), \quad \widetilde \pr-\text{\rm a.s.}
\end{align}
 \end{corollary}
\begin{proof}
For every $c>0$ and for a non-decreasing function $\tau$ as in the statement of the theorem the test \eqref{eq:test_1}
gives
\begin{align} \label{eq:test_aux}
 \widetilde \pr \left(|\widetilde X_n| \geq c \tau(n) \ \text{for infinitely many} \ n \in \N \right) = \left\{
  \begin{array}{ccc}
    0 & \quad \text{if  } \quad I_{\varphi_0}(c,\tau) < \infty,\\
    1 & \quad \text{if  } \quad I_{\varphi_0}(c,\tau) = \infty.
  \end{array}\right.
 \end{align}
The result then follows directly from \eqref{eq:test_aux}.
\end{proof}
\noindent
Below we will rewrite the integral $I_{\varphi_0}$ in a more suitable way to investigate the explicit dependence of
the result on the L\'evy triplet of the underlying process and the potential.

As a second type of result of general character we show a comparison principle. Intuitively, a more pinning potential
gives rise to a ground state which decays faster, and so the corresponding GST should fluctuate less. The following result
proves this intuition.
\begin{theorem}
\label{thm:monot}
Let $(\widetilde X^{(1)}_t)_{t \geq 0}$ and $(\widetilde X^{(2)}_t)_{t \geq 0}$ be the two GST-processes corresponding to
the ground states $\varphi_0^{(1)}$ and $\varphi_0^{(2)}$, respectively. Suppose that there exists $c_0>0$ such that for
every $c \geq c_0$ we have
\begin{align} \label{eq:monot}
\liminf_{|x| \to \infty} \frac{\varphi_0^{(1)}(cx)}{\varphi_0^{(2)}(x)} >0.
\end{align}
Then the following hold.
\begin{itemize}
\item[(1)]
For every non-decreasing function $\tau^{(2)}$ such that
 \begin{align} \label{eq:monot1}
 (0,\infty) \ni c_{\varphi_0^{(2)}} := \limsup_{n \to \infty} \frac{|\widetilde X^{(2)}_n|}{\tau^{(2)}(n)}, \quad
 \widetilde \pr-\text{\rm a.s.}
 \end{align}
it follows that
$$
\limsup_{n \to \infty} \frac{|\widetilde X^{(1)}_n|}{\tau^{(2)}(n)} = \infty, \quad \widetilde \pr-\text{\rm a.s.}
$$
\item[(2)]
If $\tau^{(2)}$ is a non-decreasing function satisfying \eqref{eq:monot1} and $\tau^{(1)}$ is a non-decreasing
function such that
\begin{align}
\label{eq:monot2}
\widetilde \pr\left(\limsup_{n \to \infty} \frac{|\widetilde X^{(1)}_n|}{\tau^{(1)}(n)} < \infty \right) > 0,
\end{align}
then also
$$
\limsup_{n \to \infty} \frac{\tau^{(1)}(n)}{\tau^{(2)}(n)} = \infty.
$$
\end{itemize}
 \end{theorem}

\begin{proof}
Suppose that condition \eqref{eq:monot} holds. Also, let \eqref{eq:monot1} be satisfied for a given non-decreasing
function $\tau^{(2)}$ and denote $c_2:= c_{\varphi_0^{(2)}}$. By a change of variable in the inner integral, for every
$c \geq c_0c_2$ and $\varepsilon \in (0, c_2)$ it follows that
\begin{align*}
I_{\varphi^{(1)}_0}(c,\tau^{(2)}) & = \int_1^{\infty} dr \int_{|x| \geq \tau^{(2)}(r)} \left(\varphi^{(1)}_0(cx)\right)^2 dx
\\ & = \left(\frac{1}{c_2-\varepsilon}\right)^d \int_1^{\infty} dr \int_{|x| \geq (c_2-\varepsilon)\tau^{(2)}(r)}
\left(\varphi^{(1)}_0\left(\frac{cx}{c_2-\varepsilon}\right)\right)^2 dx.
\end{align*}
By \eqref{eq:monot} there exist $C, R>0$ such that
$$
\varphi^{(1)}_0\left(\frac{cx}{c_2-\varepsilon}\right) \geq C \varphi^{(2)}_0\left(x\right), \quad |x| \geq R.
$$
Thus the above estimate implies
$$
I_{\varphi^{(1)}_0}(c,\tau^{(2)}) \geq C^2 \left(\frac{1}{c_2-\varepsilon}\right)^d \int_{r_0}^{\infty} dr \int_{|x|
\geq (c_2-\varepsilon)\tau^{(2)}(r)} \left(\varphi^{(2)}_0(x)\right)^2 dx,
$$
for every $r_0 \geq 1$ such that $(c_2-\varepsilon)\tau^{(2)}(r_0) \geq R$.  By \eqref{eq:monot1} and the test
\eqref{eq:test_aux}, we have $I_{\varphi^{(2)}_0}(c_2-\varepsilon,\tau^{(2)}) = \infty$ and the latter integral cannot be
convergent. This means that for every $c \geq c_0c_2$ we also have $I_{\varphi^{(1)}_0}(c,\tau^{(2)}) = \infty$. Thus
the integral test \eqref{eq:test_aux} again yields that for every $K \in \N$ such that $K \geq c_0c_2$
$$
\widetilde \pr\left(\limsup_{n \to \infty} \frac{|\widetilde X^{(1)}_n|}{\tau^{(2)}(n)} \geq K\right) = 1.
$$
This then gives
$$
\limsup_{n \to \infty} \frac{|\widetilde X^{(1)}_n|}{\tau^{(2)}(n)} = \infty, \quad \widetilde \pr-\text{\rm a.s.},
$$
which completes the proof of (1). The assertion (2) is a direct consequence of (1).
\end{proof}

%\begin{corollary}
%Suppose $V_1, V_2$ are two potentials satisfying Assumption ... Consider the envelopes $\tau_1, \tau_2$ obtained
%for the two potentials, respectively. Then $\tau_1 = \tau_2$ if and only if $V_1 \asymp V_2$.
%\end{corollary}

\begin{remark}
\rm{
We note the following in relation to the condition \eqref{eq:monot} above.
\begin{itemize}
\item[(1)]
The decay rates of the ground state eigenfunctions for confining and decaying potentials are determined by \eqref{eq:gsest}
and \eqref{eq:gsest_decay_1}-\eqref{eq:gsest_decay_2}, respectively. Thus condition \eqref{eq:monot} can be efficiently
checked for a large class of underlying L\'evy processes and potentials.

\item[(2)]
Condition \eqref{eq:monot} is immediately satisfied if the order of the decay rate of $\varphi_0^{(2)}(x)$ at  infinity
is substantially greater than that of $\varphi_0^{(1)}(x)$, e.g., when $\varphi_0^{(2)}(x) \leq c_1 e^{-c_2|x|^{\beta_1}}$ and
$\varphi_0^{(1)}(x) \geq c_3 e^{-c_4|x|^{\beta_2}}$ with $0<\beta_2 < \beta_1$, or $\varphi_0^{(1)}(x) \geq c_3 |x|^{-\gamma}$
with $\gamma > d$, for large $|x|$. For examples we refer to Section \ref{sec:Examples}.

\end{itemize}
}
\end{remark}

\section{Almost sure long time behaviour of GST-processes arising from jump-paring L\'evy processes}

\subsection{Sharp tail estimates for stationary distributions}
\noindent
First we prove a technical lemma which will be applied to derive sharp tail estimates for the stationary
distributions of the GST processes.

\begin{lemma}
\label{lem:l1}
Let $r_0 \geq 1$ and let $h :[r_0, \infty) \to (0,\infty)$ be a given function such that
\begin{itemize}
\item[(i)] $h(r) r^d \to 0$ as $r \to \infty$,
\item[(ii)] $h(r) r^{d-1} \in L^1(r_0,\infty)$,
\item[(iii)] $h \in \cC^1(r_0,\infty)$.
\end{itemize}
Consider the following conditions:
\begin{itemize}
\item[\textbf{(L)}]
There exist a non-decreasing $\cC^1$-class function $\kappa:(r_0,\infty) \to (0,\infty)$ and constants $A_1 \geq 0$ and $B_1 >0$
such that
\begin{align}
\label{ass:kappacond_low}
-r \frac{\ud}{\ud r} \frac{1}{\kappa(r)} \leq A_1
\qquad \mbox{and} \qquad
%\begin{align}
%\label{ass:logcond_low}
-r \frac{\ud}{\ud r} \log h(r) - d \leq B_1 \kappa(r),
\quad  r>r_0.
\end{align}
\item[\textbf{(U)}]
There exist a non-decreasing $\cC^1$-class function $\kappa:(r_0,\infty) \to (0,\infty)$ and constants $A_2 \geq 0$ and $B_2 >0$
such that
\begin{align} \label{ass:kappacond_up}
-r \frac{\ud}{\ud r} \frac{1}{\kappa(r)} \geq A_2
%\end{align}
\qquad \mbox{and} \qquad
%\begin{align} \label{ass:logcond_up}
-r \frac{\ud}{\ud r} \log h(r) - d \geq  B_2 \kappa(r),
\quad  r>r_0.
\end{align}
\end{itemize}
The following hold.
\begin{itemize}
\item[(1)]
Under assumptions (i)-(iii) and condition (L) we have
$$
\int_{s>r} h(s) s^{d-1} ds \leq \frac{1}{A_1+B_1} \frac{h(r) r^d}{\kappa(r)}, \quad  r>r_0.
$$
\item[(2)]
Under assumptions (i)-(iii) and condition (U) we have

$$
\int_{s>r} h(s) s^{d-1} ds \geq \frac{1}{A_2+B_2} \frac{h(r) r^d}{\kappa(r)}, \quad r>r_0.
$$
\end{itemize}
\end{lemma}

\begin{proof}
We only prove (1) as the proof of (2) goes in the same way. Since
\begin{align*}
\frac{\frac{\ud}{\ud r} \left(h(r)\frac{r^{d}}{\kappa(r)}\right)}{\frac{\ud}{\ud r} \int_{s>r} h(s)s^{d-1} ds} &  =
\frac{h^{'}(r) \frac{r^{d}}{\kappa(r)} + h(r) \frac{dr^{d-1} \kappa(r) - r^d \kappa^{'}(r)}{\kappa^2(r)}}{-h(r)r^{d-1}}
\\ &
= \frac{-r \frac{\ud}{\ud r} \log h(r) - d}{\kappa(r)} - r \frac{\ud}{\ud r} \frac{1}{\kappa(r)},
\end{align*}
for almost every $r>r_0$. By using \eqref{ass:kappacond_low} we see
that
$$
(A_1+B_1) \frac{\ud}{\ud r} \int_{s>r} h(s)s^{d-1} ds \leq \frac{\ud}{\ud r} \left(h(r)\frac{r^{d}}{\kappa(r)}\right), \quad r >r_0.
$$
Then by integrating on the two sides of the above inequality over the interval $(r, \infty)$, $r> r_0$,
and using assumptions (i)-(ii), the result follows.
\end{proof}

\subsection{The case of confining potentials} \label{sec:confining}
\noindent
In this section we consider the class of symmetric jump-paring L\'evy processes defined by Assumption \ref{ass:assnu},
and subject them to appropriate potentials.

Denote
\begin{align} \label{eq:Vul}
V_{\uppp}(x):= \sup_{y \in B(x,1)} V(y) \quad \mbox{and} \quad V_{\low}(x):= \inf_{y \in B(x,1)} V(y), \quad x \in \R^d.
\end{align}
When $V_{\uppp}(x) \asymp V_{\low}(x)$ for $|x| > R$ with some $R>0$, then we say that the values of $V$ are
\emph{almost constant on unit balls outside a bounded set} or, in short, that $V$ \emph{is almost constant on unit balls}.

We impose the following regularity condition on the potentials.

\begin{assumption}
\label{ass:pots_pin}
Let $V \in \cK^X_{\pm}$ be a confining potential, i.e. $V(x) \to \infty$ as $|x| \to \infty$.
Moreover, we assume that there exist functions $g^{\uppp}, g^{\low}: (1,\infty) \to
(0,\infty)$ such that
\begin{align}
\label{eq:gul}
g^{\uppp}(r) \asymp \left(\int_{\sS^{d-1}} \left(\frac{1}{(1 \vee V_{\uppp}(r\theta))}\right)^2  d \theta\right)^{1/2}
\quad \mbox{and} \quad
g^{\low}(r) \asymp \left(\int_{\sS^{d-1}} \left(\frac{1}{(1 \vee V_{\low}(r\theta))}\right)^2  d \theta\right)^{1/2}
\end{align}
for all $r > 1$.
\end{assumption}

\noindent
Under Assumption \ref{ass:pots_pin}, the ground state $0 < \varphi_0 \in C_{\bd}(\R^d)$ and there exist constants $C_1, C_2>0$
such that (see \cite[Th.2.4, Cor.2.2]{KL12})
\begin{align} \label{eq:gsest}
C_1 \, \frac{1 \wedge \nu(x)}{1 \vee V_{\uppp}(x)} \leq \varphi_0(x) \leq C_2 \, \frac{1 \wedge \nu(x)}{1 \vee V_{\low}(x)},
\quad x \in \R^d.
\end{align}

To make some direct computations and find the direct profile functions for paths of the processes for specific jump intensities
and given potentials $V$ it is useful to rewrite the integral test in a more explicit way.
Let $\kappa:[1,\infty) \to (0,\infty)$ be a given function. For $c>0$ and a non-decreasing function $\tau:[0,\infty) \to (0, \infty)$,
we denote
$$
I^{\uppp}_{\nu,V,\kappa}(c,\tau) = \int^{\infty} \left(g^{\uppp}(c \tau(r)) f(c \tau(r))\right)^2 \frac{\tau^{d}(r)}{\kappa(\tau(r))} dr
$$
and
$$
I^{\low}_{\nu,V,\kappa}(c,\tau) = \int^{\infty} \left(g^{\low}(c \tau(r)) f(c \tau(r))\right)^2 \frac{\tau^{d}(r)}{\kappa(\tau(r))} dr.
$$
Also, define
$$
c^{\low}_{\nu,V,\kappa}(\tau):= \inf \left\{c>0: I^{\low}_{\nu,V,\kappa}(c,\tau) < \infty\right\}
\quad \mbox{and} \quad
c^{\uppp}_{\nu,V,\kappa}(\tau):=
\sup \left\{c>0: I^{\uppp}_{\nu,V,\kappa}(c,\tau) = \infty\right\}.
$$
Since $I^{\uppp}_{\nu,V,\kappa}(c,\tau) \leq I^{\low}_{\nu,V,\kappa}(c,\tau)$ for every $c>0$ and $\tau$, we always have
$c^{\uppp}_{\nu,V,\kappa}(\tau) \leq c^{\low}_{\nu,V,\kappa}(\tau)$.

We are now ready to state the first main result in this section.

\begin{theorem}[\textbf{Integral test: jump-paring underlying process}]
\label{thm:LIL_reg_var}
Let Assumptions \ref{ass:assnu}-\ref{ass:pots_pin} hold. Assume, in addition, that the profiles $g^{\uppp}, g^{\low}$ appearing in Assumption \ref{ass:pots_pin} are $\cC^1$-class functions. Then we have the following.
\begin{itemize}
\item[(1)]
If condition (L) in Lemma \ref{lem:l1} holds for the function $h=(g^{\low} f)^2$ with $r_0=1$ and some $\kappa$,
then for every non-decreasing function $\tau : [0,\infty) \to (0,\infty)$ we have
 \begin{align} \label{eq:test_1_low}
 \widetilde \pr \left(|\widetilde X_n| \geq \tau(n) \;\;\text{\rm for infinitely many}\;\; n \in \N \right) = 0
 \quad \text{whenever} \quad I^{\low}_{\nu,V,\kappa}(1,\tau) < \infty.
 \end{align}
\item[(2)]
If condition (U) in Lemma \ref{lem:l1} holds for the function $h=(g^{\uppp} f)^2$ with $r_0=1$ and some $\kappa$,
then for every non-decreasing function $\tau : [0,\infty) \to (0,\infty)$ we have
 \begin{align} \label{eq:test_1_up}
 \widetilde \pr \left(|\widetilde X_n| \geq \tau(n) \;\;\text{\rm for infinitely many}\;\; n \in \N \right) = 1
 \quad \text{whenever} \quad I^{\uppp}_{\nu,V,\kappa}(1,\tau)  = \infty.
\end{align}
\end{itemize}
\end{theorem}

\begin{proof}
First we prove (1). Using the general Theorem \ref{thm:int_test_1}, it suffices to show that $I_{\varphi_0}(1,\tau) < \infty$
whenever $I^{\low}_{\nu,V,\kappa}(1,\tau) < \infty$. Note that when the latter integral is finite, necessarily $\tau(r)
\to \infty$ as $r \to \infty$. We have
$$
I_{\varphi_0}(1,\tau) = \int_1^{\infty}  dr \int_{|x| \geq \tau(r)} \varphi_0^2(x) dx.
$$
According to \eqref{eq:gsest}, by the fact that under our assumptions $\nu(x) \to 0$ and $V(x) \to \infty$ as $|x| \to \infty$,
there exists $R \geq 1$ such that there is a constant $C>0$ satisfying
$$
\varphi_0(x) \leq C \frac{\nu(x)}{V^{\low}(x)}, \quad |x| \geq R.
$$
Let now $r_0 > 1$ be large enough such that $\tau(r) \geq R$ for $r \geq r_0$. With this, by Assumptions \ref{ass:assnu} (i) and
\ref{ass:pots_pin},  we clearly have
\begin{align}
\int_{r_0}^{\infty} dr \int_{|x| \geq \tau(r)} \varphi_0^2(x) dx & \leq C \int_{r_0}^{\infty}  dr
\int_{|x| \geq \tau(r)} \left(\frac{\nu(x)}{V^{\low}(x)}\right)^2 dx \nonumber \\ & \leq C_1 \int_{r_0}^{\infty} dr \int_{s \geq \tau(r)}
\left(f(s) g^{\low}(s)\right)^2 s^{d-1} ds. \label{eq:aux_h}
\end{align}
To conclude, it suffices to apply Lemma \ref{lem:l1} (1) to the inner integral in \eqref{eq:aux_h}.
We first check the assumptions (i)-(iii) of this lemma for $h(s) := (f(s) g^{\low}(s))^2$, $s \geq R$. Since $g^{\low}(s)$ is bounded
for large $s$ and $f$ is the profile for the L\'evy measure far from the origin, the first two conditions (i)-(ii) follow
immediately. Moreover, $g^{\low}$ is assumed to be a $\cC^1$-class function in $(R, \infty)$.
If the same is true for $f$, then the condition (iii) holds as well and, by applying Lemma \ref{lem:l1} (1) to such $h(s)$, we get
$$
\int_{r_0}^{\infty} dr \int_{|x| \geq \tau(r)} \varphi_0^2(x) dx \leq C_2 \int_{r_0}^{\infty}
\left(f(\tau(r)) g^{\low}(\tau(r))\right)^2 \frac{\tau^{d}(r)}{\kappa(\tau(r))}  dr \leq C_3 I^{\low}_{\nu,V,\kappa}(1,\tau) <
\infty.
$$
Since the integral $\int_1^{r_0} dr \int_{|x| \geq \tau(r)} \varphi_0^2(x) dx$ is convergent, we conclude that
$I_{\varphi_0}(1,\tau) < \infty$.

On the other hand, if $f$ is not a $\cC^1$-class function,
then due to the convolution condition \eqref{eq:ass2} we can show that there is a constant $C>0$ such
that $f(s) \leq C f(s+1)$ for all $s \geq 1$ \cite[Lem.1, Lem.3 ]{KS14}. With this, we can construct a $\cC^1$-class function
$f_0$ such that $f_0(r) \asymp f(r)$, $r \geq 1$ (this can be done by putting $f_0(r):= f(r)$, for $r \in \N$,
and by $\cC^1$-interpolation). Then, the function $h(s) := (f(s) g^{\low}(s))^2$ under the integral in \eqref{eq:aux_h}
above can be replaced with $h_0(s) := (f_0(s) g^{\low}(s))^2$ to which Lemma \ref{lem:l1} (1) applies directly as above.

To see (2), it suffices to check that $I_{\varphi_0}(1,\tau) = \infty$ whenever $I^{\low}_{\nu,V,\kappa}(1,\tau) = \infty$. The
proof of this again uses the general integral test in Theorem \ref{thm:int_test_1} and similar arguments as above based on
the converse inequalities.
\end{proof}

\begin{corollary}[\textbf{Long time behaviour: jump-paring underlying process}]
\label{cor:LIL_reg_var}
Under the conditions of part (1) in Theorem \ref{thm:LIL_reg_var} it follows that
 \begin{align}
  \label{eq:LIL2_up}
 \limsup_{n \to \infty} \frac{|\widetilde X_n|}{\tau(n)} \leq c^{\low}_{\nu,V,\kappa}(\tau), \quad \widetilde \pr-\text{a.s.},
 \end{align}
and under the conditions in part (2) it follows that
\begin{align}
\label{eq:LIL2_low}
 \limsup_{n \to \infty} \frac{|\widetilde X_n|}{\tau(n)} \geq c^{\uppp}_{\nu,V,\kappa}(\tau) , \quad \widetilde \pr-\text{a.s.}
\end{align}
\end{corollary}
\begin{proof}
This is an immediate consequence of Theorem \ref{thm:LIL_reg_var}.
\end{proof}

\begin{remark}
{\rm
If $V$ is almost constant on unit balls, then for every $c>0$ and $\tau$ we have $I^{\uppp}_{\nu,V,\kappa}(c,\tau)
\asymp I^{\low}_{\nu,V,\kappa}(c,\tau)$ (i.e., both integrals are convergent or divergent at the same time) and
$c^{\low}_{\nu,V,\kappa}(\tau) = c^{\uppp}_{\nu,V,\kappa}(\tau)$. In this case the integral tests
\eqref{eq:test_1_low}-\eqref{eq:test_1_up} and the limsup resulting constants in \eqref{eq:LIL2_up}-\eqref{eq:LIL2_low}
are sharp. For more specific examples of potentials we will see that this holds in an essentially greater generality.
Moreover, if $\varphi_0$ decays polynomially at infinity (cf. Theorem \ref{thm:reg_prof_slow}), then the resulting constants
$c^{\low}_{\nu,V,\kappa}(\tau)$ and $c^{\uppp}_{\nu,V,\kappa}(\tau)$ are necessarily $0$ or $\infty$.
}
\end{remark}

In Section \ref{sec:Examples} we illustrate our Theorem \ref{thm:LIL_reg_var} and Corollary \ref{cor:LIL_reg_var} by various choices
of the L\'evy density $\nu$ and the confining potential $V$. Note that all these results apply to general non-decreasing test functions $\tau$.
This is a consequence of the sharp estimates in Lemma \ref{lem:l1}, which requires some initial smoothness of the profiles for $V$.

\subsection{Almost sure behaviour profiles for confining potentials with regular variation} \label{sec:confining_regular}
\noindent

If one is interested in constructing explicit almost sure long time behaviour profiles corresponding to specific types of L\'evy measures
and potentials, but not in the study of integral tests for general non-decreasing functions $\tau$ as in Theorem \ref{thm:LIL_reg_var},
then one can use more direct argument. Such profiles are typically strictly increasing functions and therefore one can use the Fubini
theorem instead of the tail estimates in Lemma \ref{lem:l1}.

In this section we construct directly the explicit almost sure long time behaviour profiles for the paths of the GST process in the case when $-\log (f(s) g^{\low}(s))$ and $-\log (f(s) g^{\uppp}(s))$ are asymptotically equivalent with strictly increasing regularly varying functions at infinity.

Recall that a function $\cR:(r_0,\infty) \to (0,\infty)$ is said to be regularly varying at
infinity with index $\lambda \in \R$ if
$$
\lim_{r \to \infty} \frac{\cR(sr)}{\cR(r)} = s^{\lambda}, \quad s >0,
$$
and $\cL:(r_0,\infty) \to (0,\infty)$ is called slowly varying at infinity if it is regularly varying with index
$\lambda = 0$. Every function $\cR$ regularly varying at infinity with index $\lambda \in \R$ can be represented in
the form
$$
\cR(r) = r^{\lambda} \cL(r),
$$
where $\cL$ is slowly varying at infinity. It is known that $\cL$ can be assumed to be a continuous function. %It is known that when $\lambda>0$, then $\cL$ in this representation can be assumed to be non-decreasing. For the remainder of this subsection we assume that this is the case, i.e., $\lambda>0$.
For $r > \cR(r_0)$ define
$$
\cR^{\ast}(r):= \inf\left\{s \in [r_0, \infty): \cR(s) \geq r\right\}.
$$
We have $\cR^{\ast}(r)=r^{1/\lambda} \cL^{\ast}(r)$ and $\cR^{\ast}$ is the asymptotic inverse function of $\cR$ in the
sense of the relation
\begin{align}
\label{eq:conj1}
\cR(\cR^{\ast}(r)) \approx \cR^{\ast}(\cR(r)) \approx r.
\end{align}
The notation $f(r) \approx g(r)$ means that $\lim_{r \to \infty} f(r)/g(r) = 1$. In this case the functions $f$ and $g$
are called asymptotically equivalent at infinity.
The function $\cL^{\ast}$ is slowly varying at infinity and is called the conjugate slowly varying function of $\cL$.
It is known that if $\cR^{\ast}$ is an asymptotic inverse function of $\cR$, then it is unique in the sense that if there
is another slowly varying function $\cL^{'}$ satisfying $\cR(r^{1/\lambda} \cL^{'}(r)) \approx r$, then $\cL^{'} \approx
\cL^{\ast}$. By \eqref{eq:conj1} we also have
\begin{align}
\label{eq:conj2}
\lim_{r \to \infty} (\cL^{\ast}(r))^{\lambda}\cL(r^{1/\lambda}\cL^{\ast}(r) ) = 1.
\end{align}
Recall that the function $\cR$ is called to be ultimately increasing if there exists $r_0>0$ such that $\cR$ is increasing on $(r_0,\infty)$.
For further properties and details we refer to e.g. \cite[Ch.1]{bib:ESen}.

\begin{theorem}[\textbf{Regularly varying L\'evy intensities and potentials}]
\label{thm:reg_prof}
Let Assumptions \ref{ass:assnu}-\ref{ass:pots_pin} hold and suppose that there exists $A \in (0,\infty)$ such that
for $g_1=g^{\uppp}$ and $g_2=g^{\low}$
\begin{align}
\label{eq:fluctcond_1}
\log g_i(r) + \log f(r) = - A \, \cR(r) + o(\cR(r)) \quad \text{as} \quad r \to \infty, \quad i=1,2
\end{align}
holds with $\cR(r) = r^{\lambda} \cL(r)$, where $\lambda > 0$ and $\cL:[r_0,\infty)\to(0,\infty)$ is a slowly varying
function at infinity. If $\cR$ is ultimately increasing, then
 \begin{align}
 \label{eq:res1}
 \limsup_{n \to \infty} \frac{|\widetilde X_n|}{(\log n)^{1/\lambda} \cL^{*}(\log n)}
 = \frac{1}{(2A)^{1/\lambda}}, \quad \widetilde \pr - \text{a.s.},
 \end{align}
 where $\cL^{*}$ is the conjugate slowly varying function of $\cL$.
\end{theorem}

\begin{remark}
\label{rem:rem1}
{\rm
With the settings of the above theorem, whenever
 \begin{align}
 \label{eq:techconv}
 \cL(r) \approx \cL\left(\frac{r}{(\cL(r))^{1/\lambda}}\right),
 \end{align}
we can take
 \begin{align}
 \label{eq:res2}
\cL^*(r) = \left(\cL(r^{1/\lambda})\right)^{-1/\lambda}.
 \end{align}
Indeed, under \eqref{eq:techconv} we have $\cR\left(r^{1/\lambda} \left(\cL\left(r^{1/\lambda}\right)\right)^{-1/\lambda} \right)
\approx r$, i.e., the function $r^{1/\lambda} \left(\cL\left(r^{1/\lambda}\right)\right)^{-1/\lambda}$ is the asymptotic inverse
of $\cR$ and \eqref{eq:res2} holds by the asymptotic uniqueness of $\cL^{\ast}$.
}
\end{remark}

\begin{proof} [Proof of Theorem \ref{thm:reg_prof}]
Let $\theta(r):= e^{r^{\lambda} \cL(r)}$, $r>0$. We may assume that $r_0$ is large enough such that $\cR$ is increasing and continuous on $[r_0,\infty)$. In particular, there exists an inverse function $\theta^{-1}:[\theta(r_0),\infty) \to [r_0,\infty)$. For a shorthand notation write $F_i(r):= (f(r)g_i(r))^2 r^{d-1}$.
By similar argument as in \eqref{eq:aux_h} (using two sided estimates \eqref{eq:gsest}), we have for $c>0$
\begin{align}\label{eq:reg_var_test}
C_1 I^{(1)}_{\nu,V}(c,\theta^{-1})  \leq \int_{\theta(r_0/c)}^{\infty}  dr \int_{|x| \geq c\theta^{-1}(r)} \varphi_0^2(x) dx
\leq C_2 I^{(2)}_{\nu,V}(c,\theta^{-1}),
\end{align}
where
\begin{align*}
I^{(i)}_{\nu,V}(c,\theta^{-1}) = \int_{\theta(r_0/c)}^{\infty} dr \int_{s \geq c\theta^{-1}(r)}  F_i(s) ds, \quad i=1, 2,
\end{align*}
and the constants $C_1, C_2$ do not depend on $c$ and $\theta$. Moreover, by Fubini's theorem,
$$
I^{(i)}_{\nu,V}(c,\theta^{-1}) = \int_{r_0}^{\infty} \theta(r/c) F_i(r) dr \quad i =1, 2.
$$
It follows from \eqref{eq:fluctcond_1} that for every $\varepsilon \in (0,1)$ there is $r_{\varepsilon} >0$ such that for
all $r>r_{\varepsilon}$
\begin{eqnarray}
\label{eq:equ1}
F_1(r) \geq  e^{- 2(1+\varepsilon)A r^{\lambda} \cL(r)} \quad \mbox{and} \quad
F_2(r) \leq  e^{- 2(1-\varepsilon)A r^{\lambda} \cL(r)}.
\end{eqnarray}
With this we have for every $c>0$
$$
\theta(r/c) F_2(r) \leq \exp\left(\left(\frac{1}{c^{\lambda}} \frac{\cL(r/c)}{\cL(r)} - 2 (1-\varepsilon) A\right) r^{\lambda}
\cL(r)\right), \quad r > r_{\varepsilon},
$$
and
$$
\theta(r/c) F_1(r) \geq \exp\left(\left(\frac{1}{c^{\lambda}} \frac{\cL(r/c)}{\cL(r)} - 2 (1+\varepsilon) A\right) r^{\lambda}
\cL(r)\right), \quad r > r_{\varepsilon}.
$$
Therefore, by \eqref{eq:equ1} and by slow variation of $\cL$, for every $c > (2A)^{-1/\lambda}$ there exist $\varepsilon \in (0,1)$
and $R>0$ such that
$$
\frac{1}{c^{\lambda}} \frac{\cL(r/c)}{\cL(r)} - 2 (1+\varepsilon) A < 0, \quad r>R.
$$
Hence $I^{(2)}_{\nu,V}(c,\theta^{-1}) < \infty$ whenever $c > (2A)^{-1/\lambda}$. Due to \eqref{eq:reg_var_test} also $I_{\varphi_0}(c,\theta^{-1}) < \infty$ for this range of $c$. By similar argument we can also show that
$I^{(1)}_{\nu,V}(c,\theta^{-1}) = \infty$ (and, therefore, $I_{\varphi_0}(c,\theta^{-1}) = \infty$), for every $c < (2A)^{-1/\lambda}$. We  then have $c_{\varphi_0}(\tau) = (2A)^{-1/\lambda}$ and, by Corollary \ref{cor:int_test_1} with $\tau=\theta^{-1}$,
we finally get
$$
\limsup_{n \to \infty} \frac{|\widetilde X_n|}{\theta^{-1}(n)} = \frac{1}{(2A)^{1/\lambda}}, \quad \widetilde \pr - \text{a.s.}
$$
To complete the proof it suffices to observe that by asymptotic uniqueness of $\cR^*$ it follows that $\cR^*(\log r) \approx
\theta^{-1}(r)$.
\end{proof}

The next theorem involves the L\'evy intensities and potentials of slow variation at infinity. For convenience, denote
$k$-fold iterated logarithm by $\log_k$.
\begin{theorem}[\textbf{Slowly varying L\'evy intensities and potentials}]
\label{thm:reg_prof_slow}
Let Assumptions \ref{ass:assnu}-\ref{ass:pots_pin} hold and suppose that there exist $\gamma \in [d,\infty)$, $l \in \N$
and $\beta_1,...,\beta_l \in \R$ such that for $g_1=g^{\uppp}$ and $g_2=g^{\low}$ we have
\begin{align}
\label{eq:fluctcond_1_slow}
f(r) g_i(r) \asymp r^{-\gamma}(\log r)^{\beta_1} (\log_2 r)^{\beta_2} \cdots (\log_l r)^{\beta_l}
\quad \text{as} \quad r \to \infty, \quad i=1,2.
\end{align}
For natural numbers $k \geq l$ and any $\delta>0$ denote
$$
\tau_k(r):= r^{\frac{1}{2\gamma-d}} \left((\log r)^{2\theta_1+1} (\log_2 r)^{2\theta_2+1} \cdots
(\log_k r)^{2\theta_k+\delta}\right)^{\frac{1}{2\gamma-d}},
$$
where $\theta_i = \beta_i$ for $1 \leq i \leq l$ and $\theta_i = 0$ for $l < i \leq k$, whenever $k >l$.
Then for every $k \geq l$ we have $\widetilde \pr$-almost surely
 \begin{align}
 \label{eq:res1_slow}
 \limsup_{n \to \infty} \frac{|\widetilde X_n|}{\tau_k(n)} = \left\{
  \begin{array}{ccc}
    0 & \quad \text{\rm if} \quad  \delta > 1,\\
    \infty & \quad \text{\rm if} \quad \delta \leq 1.
  \end{array}\right.
 \end{align}
\end{theorem}

\begin{proof}
Let $\gamma \in [d,\infty)$, $l \in \N$ and $\beta_1,...,\beta_l \in \R$ be given and let $F_i(r):=
(f(r)g_i(r))^2 r^{d-1}$. Fix $k \geq l$ and for $\delta>0$ consider the function
$$
\vartheta(r):= r^{2\gamma-d} (\log r)^{-2\theta_1-1} (\log_2 r)^{-2\theta_2-1} \cdots
(\log_k r)^{-2\theta_k-\delta}, \quad r > \exp_k e,
$$
where $\exp_k$ denotes the $k$-fold iterated exponential function, $\theta_i = \beta_i$ for $1 \leq i \leq l$,
and $\theta_i = 0$ for $l < i \leq k$, whenever $k >l$. Clearly, $\vartheta$ is continuous on $(\exp_k e,\infty)$.
We can also check that there exists $R=R(k,\gamma, \theta_1,...,\theta_k,\delta) \geq \exp_k e$ such that
$\vartheta$ is an increasing function on $(R,\infty)$. Similarly as in the previous proof, it is enough to consider the integrals
$$
I^{(i)}_{\nu,V}(c,\vartheta^{-1}) = \int_{R}^{\infty} \vartheta(r/c) F_i(r) dr, \quad i=1, 2, \quad c>0.
$$
By \eqref{eq:fluctcond_1_slow}, for every $c>0$ there is $R_c \geq R$ such that for $i=1,2$ and every $r > R_c$
we have
$$
\vartheta(r/c) F_i(r) \asymp r^{-1} (\log r)^{-1} (\log_2 r)^{-1} \cdots (\log_k r)^{-\delta}.
$$
From this we see that $I^{(2)}_{\nu,V}(c,\theta^{-1}) < \infty$ for every $c > 0$, whenever $\delta>1$, and
similarly, $I^{(1)}_{\nu,V}(c,\theta^{-1}) = \infty$ for every $c>0$, whenever $\delta \in (0,1]$. By Corollary
\ref{cor:int_test_1} with $\tau=\vartheta^{-1}$, we finally get that $\widetilde \pr$-almost surely
$$
 \limsup_{n \to \infty} \frac{|\widetilde X_n|}{\vartheta^{-1}(n)} = \left\{
  \begin{array}{ccc}
    0 & \quad \text{if} \quad  \delta > 1,\\
    \infty & \quad \text{if} \quad \delta \leq 1.
  \end{array}\right.
$$
Similarly as in the previous theorem, it suffices to check that $\vartheta^{-1}(r) \approx \tau_k(r)$. Since
$\vartheta(r)$ is regularly varying with index $2\gamma-d$, its asymptotic inverse function is of the form
$r^{1/(2\gamma-d)} \cL^*(r)$. Hence by asymptotic uniqueness of $\cL^*(r)$ and by Remark \ref{rem:rem1} we obtain
$\vartheta(r) \approx \tau_k(r)$.
\end{proof}

\subsection{The case of decaying potentials} \label{sec:decaying}

Next we consider potentials satisfying the following condition.

\begin{assumption}
\label{ass:pots_decay}
Let $V \in \cK^X_{\pm}$ be a decaying potential, i.e. $V(x) \to 0$ as $|x| \to \infty$, and let $\lambda_0<0$ be an
isolated eigenvalue of $H$.
\end{assumption}
\noindent

As shown in \cite{KL17} (see also \cite{CMS}), the fall-off of $\varphi_0$ depends now on the rate of the decay of
$\nu$ at infinity and the distance of $\lambda_0$ from the essential spectrum of $H$. Typically, the following
three different situations may occur:

\begin{itemize}
\item[(1)]
If the L\'evy density $\nu(x)$ is localized strictly sub-exponentially at infinity (cf. \cite[Th. 4.1 and 4.3]{KL17}),
then
\begin{align} \label{eq:gsest_decay_1}
C_1 \, (1 \wedge \nu(x)) \leq \varphi_0(x) \leq C_2 \, (1 \wedge \nu(x)),
\end{align}
with constants $C_1=C_1(X, \lambda_0)$ and $C_2=C_2(X, \lambda_0)$ (note that the estimates \eqref{eq:gsest_decay_1}
depend on $\lambda_0<0$ only via the multiplicative constants $C_1$ and $C_2$).

\item[(2)]
There exists $\eta_0=\eta_0(X) > 0$, independent of $V$ (cf. \cite[(3.3)]{KL17}), such that if $\lambda_0 \in
(-\infty,-\eta_0)$ (i.e. $\lambda_0$ is sufficiently low-lying eigenvalue) and the L\'evy density $\nu(x)$ is localized
exponentially at infinity, then the estimates \eqref{eq:gsest_decay_1} holds as well (cf. \cite[Th. 4.2]{KL17}).

\item[(3)]
If $\lambda_0 \in [-\eta_0,0)$ and the L\'evy density $\nu(x)$ is localized exponentially at infinity, then there is a
constant $\theta>0$ such that for every $\varepsilon \in (0,1)$ there exists a constant $C$ such that
\begin{align} \label{eq:gsest_decay_2}
\varphi_0(x) \geq C \, \left( e^{-\theta \, \sqrt{|\lambda_0| + \varepsilon} \, |x|} \vee (1 \wedge \nu(x))\right).
\end{align}
\end{itemize}
\noindent
We refer the reader to \cite[Sec. 4.3-4.4]{KL17} for further discussion.

We now analyze the cases (1)-(2) and (3) separately, and illustrate them by specific examples. For simplicity, in our
results below we refer directly to the estimates \eqref{eq:gsest_decay_1}-\eqref{eq:gsest_decay_2}.

As in the previous subsection, let $\kappa:[1,\infty) \to (0,\infty)$ be a given function. For $c>0$, a non-decreasing
function $\tau:[0,\infty) \to (0, \infty)$ and $\varepsilon \in (0,1)$, we denote
$$
I_{\nu,\kappa}(c,\tau) = \int^{\infty} \left(f(c \tau(r))\right)^2 \frac{\tau^{d}(r)}{\kappa(\tau(r))} dr \quad \text{and}
\quad I^{\varepsilon}_{\lambda_0,\kappa}(c,\tau) = \int^{\infty} e^{-2c\theta \, \sqrt{|\lambda_0| + \varepsilon} \, \tau(r)}
\tau^{d-1}(r) dr
$$
Also, let
$$
c_{\nu,\kappa}(\tau):= \inf \left\{c>0: I_{\nu,\kappa}(c,\tau) < \infty\right\}
\quad \text{and} \quad
c^{\varepsilon}_{\lambda_0,\kappa}(\tau):= \sup \left\{c>0: I^{\varepsilon}_{\lambda_0,\kappa}(c,\tau) = \infty\right\}.
$$

We are now ready to state the version of Theorem \ref{thm:LIL_reg_var} for decaying potentials in cases (1)-(2) and (3) above.

\begin{theorem}[\textbf{Integral test: jump-paring underlying process}]
\label{thm:LIL_reg_var_dec}
Let Assumptions \ref{ass:assnu} and \ref{ass:pots_decay} hold. Then we have the following.
\begin{itemize}
\item[(1)]
If \eqref{eq:gsest_decay_1} hold and if the conditions (L) and (U) in Lemma \ref{lem:l1} hold for the function $h=f^2$
with $r_0=1$ and some $\kappa$, then for every non-decreasing function $\tau : [0,\infty) \to (0,\infty)$ we have
 \begin{align} \label{eq:test_1_decay}
\widetilde \pr \left(|\widetilde X_n| \geq \tau(n) \  \text{\rm for infinitely many} \ n \in \N \right) = \left\{
  \begin{array}{ccc}
    0 & \quad \text{\rm if} \quad I_{\nu,\kappa}(1,\tau) < \infty,\\
    1 & \quad \text{\rm if} \quad I_{\nu,\kappa}(1,\tau) = \infty.
  \end{array}\right.
\end{align}	
\item[(2)]
If \eqref{eq:gsest_decay_2} holds, then for every non-decreasing function $\tau : [0,\infty) \to (0,\infty)$ we have
 \begin{align} \label{eq:test_2_decay}
\widetilde \pr \left(|\widetilde X_n| \geq \tau(n) \  \text{\rm for infinitely many} \ n \in \N \right) = 1,
 \end{align}
whenever $I^{\varepsilon}_{\lambda_0,\kappa}(1,\tau) = \infty$, for some $\varepsilon \in (0,1)$.
\end{itemize}
\end{theorem}

\begin{proof}
We can use the same arguments as in the proof of Theorem \ref{thm:LIL_reg_var}. The difference is that now the decay of
the ground state $\varphi_0$ at infinity is determined by \eqref{eq:gsest_decay_1} and  \eqref{eq:gsest_decay_2}, respectively.
Thus we take $h(r) = f^2(r)$ and $h(r) = e^{-2\theta \, \sqrt{|\lambda_0| + \varepsilon} \, r}$ in parts (1) and (2) above.
\end{proof}

\begin{corollary}[\textbf{Long time behaviour: jump-paring underlying process}]
\label{cor:LIL_reg_var_dec}
Under the assumptions of Theorem \ref{thm:LIL_reg_var_dec} we have
\begin{align}
  \label{eq:LIL2_up_decay}
 \limsup_{n \to \infty} \frac{|\widetilde X_n|}{\tau(n)} = c_{\nu,\kappa}(\tau), \quad \widetilde \pr-\text{a.s.},
\end{align}
under \eqref{eq:gsest_decay_1} and, for every $\varepsilon \in (0,1)$,
\begin{align}
  \label{eq:LIL2_up_decay_lower}
 \limsup_{n \to \infty} \frac{|\widetilde X_n|}{\tau(n)} \geq c^{\varepsilon}_{\lambda_0,\kappa}(\tau), \quad \widetilde
 \pr-\text{a.s.},
\end{align}
under \eqref{eq:gsest_decay_2}.
\end{corollary}
\begin{proof}
This is an immediate consequence of Theorem \ref{thm:LIL_reg_var_dec}.
\end{proof}
In the case of decaying potentials, we may also formulate versions of Theorems \ref{thm:reg_prof}-\ref{thm:reg_prof_slow}.
Since the proofs of these results are similar, we leave them to the reader.

\begin{corollary}[\textbf{Regularly varying L\'evy intensities}]
\label{cor:reg_prof_dec}
Let Assumptions \ref{ass:assnu} and \ref{ass:pots_decay} hold.
\begin{itemize}
\item[(1)]
If \eqref{eq:gsest_decay_1} holds and there exists $A \in (0,\infty)$ such that
\begin{align}
\label{eq:fluctcond_1_dec }
\log f(r) = - A \, \cR(r) + o(\cR(r)) \quad \text{as} \quad r \to \infty,
\end{align}
with increasing $\cR(r) = r^{\lambda} \cL(r)$, where $\lambda \in (0,1]$ and $\cL:[r_0,\infty)\to(0,\infty)$ is a slowly varying
function at infinity, then
 \begin{align}
 \label{eq:res1_dec}
 \limsup_{n \to \infty} \frac{|\widetilde X_n|}{(\log n)^{1/\lambda} \cL^{*}(\log n)}
 = \frac{1}{(2A)^{1/\lambda}}, \quad \widetilde \pr - \text{a.s.},
 \end{align}
 where $\cL^{*}$ is the conjugate slowly varying function of $\cL$ (cf. Remark \ref{rem:rem1}).
\item[(2)]
If \eqref{eq:gsest_decay_2} holds, then
 \begin{align}
 \label{eq:res12_dec}
 \limsup_{n \to \infty} \frac{|\widetilde X_n|}{\log n} \geq \frac{1}{2 \theta \sqrt{|\lambda_0|}},
 \quad \widetilde \pr - \text{a.s.}
 \end{align}
\end{itemize}
\end{corollary}

\begin{corollary}[\textbf{Slowly varying L\'evy intensities}]
\label{cor:reg_prof_slow_dec}
Let Assumptions \ref{ass:assnu} and \ref{ass:pots_decay} hold and suppose that there exist $\gamma \in [d,\infty)$, $l \in \N$
and $\beta_1,...,\beta_l \in \R$ such that
\begin{align}
\label{eq:fluctcond_1_slow_dec}
f(r) \asymp r^{-\gamma}(\log r)^{\beta_1} (\log_2 r)^{\beta_2} \cdots (\log_l r)^{\beta_l}
\quad \text{as} \quad r \to \infty.
\end{align}
For natural numbers $k \geq l$ and any $\delta>0$ denote
$$
\tau_k(r):= r^{\frac{1}{2\gamma-d}} \left((\log r)^{2\theta_1+1} (\log_2 r)^{2\theta_2+1} \cdots
(\log_k r)^{2\theta_k+\delta}\right)^{\frac{1}{2\gamma-d}},
$$
where $\theta_i = \beta_i$ for $1 \leq i \leq l$ and $\theta_i = 0$ for $l < i \leq k$, whenever $k >l$.
Then for every $k \geq l$ we have $\widetilde \pr$-almost surely
 \begin{align}
 \label{eq:res1_slow}
 \limsup_{n \to \infty} \frac{|\widetilde X_n|}{\tau_k(n)} = \left\{
  \begin{array}{ccc}
    0 & \quad \text{\rm if} \quad  \delta > 1,\\
    \infty & \quad \text{\rm if} \quad \delta \leq 1.
  \end{array}\right.
 \end{align}
\end{corollary}

\section{\rm \textbf{Examples }} \label{sec:Examples}
Now we discuss the asymptotic behaviour of paths of ground state-transformed processes with underlying L\'evy processes
having absolutely continuous L\'evy measures with densities $\nu(x) \asymp f(|x|)$ such that
\begin{align} \label{eq:ex_L_dens}
f(r) = \1_{(0,1]}(r) \, r^{-d-\alpha} + \1_{(1,\infty)}(r) \, e^{-\mu r^{\beta}} r^{-\gamma}, \quad r >0,
\end{align}
where $d \geq 1$, $\alpha \in (0,2)$, $\mu \geq 0$, $\beta \geq 0$ and $\gamma \geq 0$. As proven in \cite[Prop. 2]{KS14},
condition \eqref{eq:ass2} holds exactly in three cases:
\begin{itemize}
\item[(L1)] $\mu =0$ and $\gamma > d$
\item[(L2)] $\mu > 0$, $\beta \in (0,1)$ and $\gamma \geq 0$
\item[(L3)] $\mu >0$, $\beta= 1$ and $\gamma > (d+1)/2$.
\end{itemize}
All the other assumptions hold as well. Notice that this choice of the profile $f$ leads naturally to the following important classes
of the underlying L\'evy processes. In particular, (L1) includes the \emph{isotropic $\alpha$-stable processes} ($\gamma = d+\alpha$)
and \emph{layered $\alpha$-stable processes} ($\gamma > d+\alpha$), and (L3) includes \emph{relativistic $\alpha$-stable processes}
($\mu = m^{1/\alpha}$, $\gamma = (d+1+\alpha)/2$, for $m >0$) and \emph{tempered stable processes} ($\mu >0$, $\gamma = d+\alpha$).

First we consider confining potentials $V(x) \asymp g(|x|)$ with
\begin{align} \label{eq:ex_V}
g(r) = e^{\eta r^{\vartheta}} r^{\rho} \log(1+r)^{\sigma}, \quad r \geq 0,
\end{align}
where $\eta, \vartheta, \rho, \sigma \geq 0$ are chosen in a way that $g(r) \to \infty$ as $r \to \infty$.
Observe that with this choice of $f$ and $g$ the integral tests in Theorem \ref{thm:LIL_reg_var} hold with
$I^{\uppp}_{\nu,V,\kappa}(1,\tau) = I^{\low}_{\nu,V,\kappa}(1,\tau) = I(1,\tau)$, where
$$
I(1,\tau) := \int^{\infty} \frac{e^{-2\mu \tau(r)^{\beta} - 2\eta \tau(r)^{\vartheta}}
\tau(r)^{d-2(\gamma+\rho)}}{\log(1+\tau(r))^{\sigma} \kappa(\tau(r))} \, dr
$$
and $\kappa(r) = r^{\beta \vee \vartheta}$ if $\mu, \beta >0$, or $\eta, \vartheta >0$ and $\kappa(r) \equiv {\rm const}$ otherwise.

Moreover, the following illustrates the facts established in Corollary \ref{cor:LIL_reg_var} and Theorems
\ref{thm:reg_prof}-\ref{thm:reg_prof_slow}, highlighting the parameters of $\nu$ and $V$ that determine the long time behaviour.

\begin{example} \label{ex:ex_last_1}
{\rm
Suppose the profiles $f$ and $g$ for the density of the L\'evy measure $\nu$ and the potential $V$
are given by \eqref{eq:ex_L_dens} and \eqref{eq:ex_V}, respectively. Then we have the following.
\begin{itemize}
\item[(1)]
\texttt{Stretched exponential and exponential jump intensity:}  Let (L2) or (L3) hold.
\begin{itemize}
\item[(1.1)]
If $\eta>0$ and $\vartheta>\beta$, then
$$
\limsup_{n \to \infty} \frac{|\widetilde X_n|}{(\log n)^{1/\vartheta}} = \frac{1}{(2\eta)^{1/\vartheta}},
\quad \widetilde\pr - a.s.
$$
\item[(1.2)]
If $\eta>0$ and $\vartheta=\beta$, then
$$
\limsup_{n \to \infty} \frac{|\widetilde X_n|}{(\log n)^{1/\vartheta}} = \frac{1}{(2(\mu+\eta))^{1/\vartheta}},
\quad \widetilde\pr - a.s.
$$
\item[(1.3)]
If $\eta>0$ and $\beta>\vartheta$ or $\eta=0$, then
$$
\limsup_{n \to \infty} \frac{|\widetilde X_n|}{(\log n)^{1/\beta}} = \frac{1}{(2\mu)^{1/\beta}},
\quad \widetilde\pr - a.s.
$$
\end{itemize}

\vspace{0.1cm}
\item[(2)]
\texttt{Polynomial jump intensity:}  Let (L1) hold.
\begin{itemize}
\item[(2.1)]
If $\eta, \vartheta>0$, then
$$
\limsup_{n \to \infty} \frac{|\widetilde X_n|}{(\log n)^{1/\vartheta}} = \frac{1}{(2\eta)^{1/\vartheta}},
\quad \widetilde\pr - a.s.
$$
\item[(2.2)]
If $\eta=0$ and $\rho \geq 0$, then
$$
\limsup_{n \to \infty} \frac{|\widetilde X_n|}{\left(n (\log n)^{2\sigma+\delta}\right)^{\frac{1}{2(\gamma+\rho)+d}}} =
\left\{
  \begin{array}{ccc}
    0 & \quad \text{if} \quad  \delta > 1,\\
    \infty & \quad \text{if} \quad \delta \leq 1,
  \end{array}\right.
\quad \widetilde\pr - a.s.
$$
\end{itemize}
\end{itemize}
}
\end{example}
\noindent
Note that Example \ref{ex:ex_last_1} (2.2) applies directly to the fractional GST-processes related to $H=
(-\Delta)^{\alpha/2}+V$, where $\alpha \in (0,2)$ and $V$ is a confining potential from Example \ref{confipot}.
If $V(x)=|x|^{2n}$, $n \in \N$ (harmonic and anharmonic oscillators), then this result holds with $\gamma = d+
\alpha$, $\rho = 2n$ and $\sigma=0$. If $V$ is a double or multiple potential well, then a similar result holds
with a suitable $\rho$.

Next we illustrate our results obtained in Section \ref{sec:decaying} for decaying potentials. Corollaries
\ref{cor:reg_prof_dec}-\ref{cor:reg_prof_slow_dec} imply the following.

\begin{example} \label{ex:ex_last_2}
{\rm
Suppose the profile $f$ for the density of the L\'evy measure $\nu$ is given by \eqref{eq:ex_L_dens}, and $V$ is a
decaying potential such that Assumption \ref{ass:pots_decay} holds. Then we have the following.
\begin{itemize}
\item[(1)] If (L1) holds, then
$$
\limsup_{n \to \infty} \frac{|\widetilde X_n|}{\left(n (\log n)^{\delta}\right)^{\frac{1}{2\gamma+d}}} = \left\{
  \begin{array}{ccc}
    0 & \quad \text{if} \quad  \delta > 1,\\
    \infty & \quad \text{if} \quad \delta \leq 1,
  \end{array}\right.
\quad \widetilde\pr - a.s.
$$
\item[(2)] If (L2) holds, then
$$
\limsup_{n \to \infty} \frac{|\widetilde X_n|}{(\log n)^{1/\beta}} = \frac{1}{(2\mu)^{1/\beta}},
\quad \widetilde\pr - a.s.
$$
\item[(3)] If (L3) holds and the ground state eigenvalue $\lambda_0<0$ is sufficiently low-lying
(so that \eqref{eq:gsest_decay_1} holds), then
$$
\limsup_{n \to \infty} \frac{|\widetilde X_n|}{\log n} = \frac{1}{2\mu},
\quad \widetilde\pr - a.s.
$$
\item[(4)] If (L3) holds and the ground state eigenvalue $\lambda_0<0$ is close to zero
(so that \eqref{eq:gsest_decay_2} holds), then
 \begin{align*}
 \limsup_{n \to \infty} \frac{|\widetilde X_n|}{\log n}
 \geq \frac{1}{2 \theta \sqrt{|\lambda_0|}}, \quad \widetilde \pr - \text{a.s.}
 \end{align*}
\end{itemize}
}
\end{example}
\noindent
Recall that some classes of decaying potentials of special importance are listed in Example \ref{decaypot}, to which
these results can be applied.

Interestingly, our general results obtained in Section \ref{sec:general} apply directly to diffusive GST-proceses as well.
Indeed, in many cases the behaviour of the ground state $\varphi_0$ at infinity is known explicitly and we can analyze
the test integrals $I_{\varphi_0}(1,\tau)$ by similar methods as in Sections \ref{sec:confining}-\ref{sec:confining_regular}.
For instance, this can be done for some of the GST-Brownian motions.
Below we give the limsup-almost sure behaviour profiles for the two important models discussed in Example \ref{GST-BM}.
The details are left to the reader.

\begin{example}[\textbf{Almost sure behaviour profiles for GST Brownian motion}]
{\rm
\noindent
\begin{itemize}
\item[(1)]
\emph{Ornstein-Uhlenbeck process:} If $(\widetilde X_t)_{t \geq 0}$ is a GST-process described in Example \ref{GST-BM} (1),
then
$$
\limsup_{n \to \infty} \frac{|\widetilde X_n|}{\sqrt{\log n}} = \frac{1}{\sqrt{\gamma}},
\quad \widetilde\pr - {\rm a.s.}
$$
This result is well-known, and reproduced by our results above.
\vspace{0.1cm}
\item[(2)]
\emph{Brownian motion in a finite potential well:} If $(\widetilde X_t)_{t \geq 0}$ is a GST-process described in Example
\ref{GST-BM} (2), then
$$
\limsup_{n \to \infty} \frac{|\widetilde X_n|}{\log n} = \frac{1}{2\sqrt{2 |\lambda_0|}},
\quad \widetilde\pr - {\rm a.s.}
$$
Note that the almost sure asymptotics for this model is close to that obtained for the jump type GST processes constructed
for decaying potentials in the case when the ground state eigenvalue $\lambda_0$ is close to zero (cf. Example \ref{ex:ex_last_2}
(4)).
\end{itemize}
}
\end{example}

\bigskip
\noindent
\textbf{Acknowledgments:} JL is pleased to thank Nikolai Leonenko for supplying reference \cite{ALS}.


\begin{thebibliography}{00}
\bibitem{A1}
J.A.D. Appleby, H. Wu: Solutions of stochastic differential equations obeying the law of the iterated logarithm,
with applications to financial markets, \emph{Electronic J. Probab.} \textbf{14}, 912-959, 2009

\bibitem{A2}
J.A.D. Appleby, X. Mao, H.Z. Wu: On the size of the largest pathwise deviations of stochastic functional differential
equations, \emph{Nonlinear Stud.}, to appear

\bibitem{ALS}
F. Avram, N.N. Leonenko and N. \v{S}uvak: On spectral analysis of heavy-tailed Kolmogorov-Pearson diffusions,
\emph{Markov Proc. Relat. Fields} \textbf{19}, 249-298, 2013

\bibitem{B}
J. Bertoin: \emph{L\'evy Processes}, Cambridge University Press, 1996

\bibitem{BL}
V. Betz, J. L\H{o}rinczi: Uniqueness of Gibbs measures relative to Brownian motion,
\emph{Ann. I. H. Poincar\'e} \textbf{39}, 877-889, 2003

\bibitem{B}
L. Breiman: A delicate law of the iterated logarithm for non-decreasing stable processes,
\emph{Ann. Math. Stat.} \textbf{39}, 1814-1824, 1968; errata \textbf{41}, 1126, 1970

\bibitem{CMS}
R. Carmona, W.C. Masters, B. Simon:
Relativistic Schr\"odinger operators: asymptotic behaviour of the eigenfunctions, \emph{J. Funct. Anal.}
\textbf{91}, 117-142, 1990

\bibitem{CsR}
M. Cs\"org\H{o}, P. R\'ev\'esz: \emph{Strong Approximations in Probability and Statistics},
Academic Press, 1981

\bibitem{DS}
E.B. Davies, B. Simon: Ultracontractivity and heat kernels for Schr\"odingers operators and
Dirichlet Laplacians, \emph{J. Funct. Anal.} \textbf{59}, 335--395, 1984

\bibitem{DL}
S.O. Durugo and J. L\H orinczi:
Spectral properties of the massless relativistic quartic oscillator, \emph{J. Diff. Equations}
\textbf{264}, 3775-3809, 2018

\bibitem{Fr}
D. Freedman: \emph{Brownian Motion and Diffusion}, Holden-Day, 1971, 2nd ed. Springer, 1982

\bibitem{F}
B. Fristedt: Sample functions of stochastic processes with stationary independent increments, in:
\emph{Advances in Probability}, vol. 3, P. Ney, S. Port (eds.) pp. 241-396, Marcel Dekker, 1974.

\bibitem{G1}
S. Gheryani, F. Hiroshima, J. L\H{o}rinczi, A. Majid and H. Ouerdiane:
$P(\phi)_{1}$-process for the spin-boson model and a functional central limit theorem for associated additive
functionals, \emph{Stochastics} \textbf{89}, 1104-1115, 2017

\bibitem{G2}
S. Gheryani, F. Hiroshima, J. L\H{o}rinczi, A. Majid and H. Ouerdiane:
Functional central limit theorems and $P(\phi)_{1}$-processes for the classical and relativistic Nelson models,
arXiv:1605.01791, 2016 (under review)


\bibitem{G}
B.V. Gnedenko:
O roste odnorodnykh sluchainykh protsessov s nezavisimymi prirashcheniyami (Sur la croissance des processus
stochastiques homog\`enes \`a accroissements ind\'ependants; in Russian with a summary in French),
\emph{Izv. Akad. Nauk SSSR Ser. Mat.} \textbf{7}, 89-110, 1943

\bibitem{KL12}
K. Kaleta, J. L\H{o}rinczi:
Fractional $P(\phi)_1$-processes and Gibbs measures,\emph{ Stoch. Proc. Appl.} \textbf{122}, 3580-3617,
2012

\bibitem{KL15}
K. Kaleta, J. L\H{o}rinczi:
Pointwise estimates of the eigenfunctions and intrinsic ultracontractivity-type properties of Feynman-Kac
semigroups for a class of L\'evy processes, \emph{Ann. Probab.} \textbf{43}, 1350-1398, 2015

\bibitem{KL17}
K. Kaleta, J. L\H{o}rinczi:
Fall-off of eigenfunctions for non-local Schr\"odinger operators with decaying potentials, \emph{Potential Anal.}
\textbf{46}, 647-688, 2017

\bibitem{KSch18}
K. Kaleta, R.L. Schilling:
Progressive intrinsic ultracontractivity and heat kernel estimates for non-local Schr\"odinger operators, preprint 2018.

\bibitem{KS14}
K. Kaleta, P. Sztonyk:
Small time sharp bounds for kernels of convolution semigroups, \emph{J. Anal. Math.} \textbf{132}, 355-394, 2017

\bibitem{K}
S. Keprta: \emph{Integral Tests for Brownian Motions and Some Related Processes}, PhD thesis, Carleton
University, Ottawa, 1997

\bibitem{Kha}
A. Ya. Khintchine: \emph{Asymptotische Gesetze der Wahrscheinlichkeitsrechnung}, Kap. 5, Springer, 1933

\bibitem{Khb}
A. Ya. Khintchine: Dvo teoremy o stokhasticheskikh protssesakh s odnotipnymi prirashchenyami (Zwei S\"atze \"uber
stochastische Prozesse mit stabilen Verteilungen; in Russian with a summary in German), \emph{Mat. Sbornik}
\textbf{3}, 577-584, 1938

\bibitem{K}
F. K\"uhn:
L\'evy-Type Processes: Moments, Construction and Heat Kernel Estimates, L\'evy Matters vol. VI, Lecture
Notes in Mathematics \textbf{2187}, Springer, 2017

\bibitem{LM}
J. L\H{o}rinczi, J. Ma{\l}ecki:
Spectral properties of the massless relativistic harmonic oscillator, \emph{J. Diff. Equations} \textbf{253},
2846-2871, 2012

\bibitem{LHB11}
J. L\H{o}rinczi, F. Hiroshima, V. Betz: \emph{Feynman-Kac-Type Theorems and Gibbs Measures on Path Space. With
Applications to Rigorous Quantum Field Theory}, de Gruyter Studies in Mathematics \textbf{34}, Walter de Gruyter,
2011, 2nd ed. 2018/19

\bibitem{LY}
J. L\H orinczi and X. Yang:
Multifractal properties of sample paths of ground state-transformed jump processes, arXiv:1705.00551, 2017
(under review)

\bibitem{Ma}
X. Mao: \emph{Stochastic Differential Equations and Applications}, Horwood Publishing, 2nd ed. 2008

\bibitem{Mo}
M. Motoo: Proof of the law of iterated logarithm through diffusion equation, \emph{Ann. Inst. Statist. Math.}
\textbf{10}, 21-28, 1958

\bibitem{P}
I. Petrowsky: Zur ersten Randwertaufgabe der W\"armeleitungsgleichung, \emph{Comp. Math.} \textbf{1}, 383-419,
1935

\bibitem{PT}
W.E. Pruitt, S.J. Taylor: Sample path properties of processes with stable components, \emph{Z. Wahrsch. verw. Gebiete}
\textbf{12}, 267-289, 1969

\bibitem{ReS}
M. Reed, B. Simon: \emph{Methods of Modern Mathematical Physics, 1. Functional Analysis}, Academic Press, 1980

\bibitem{R}
P. R\'ev\'esz: \emph{Random Walk in Random and Non-Random Environments}, World Scientific, 1990, 2nd ed. 2005


\bibitem{RS}
J. Rosen, B. Simon: Fluctuations in $P(\phi)_1$-processes. {\it Ann. Probab.} {\bf 4}, 155-174, 1976


\bibitem{Sa}
N. Sandri\'c:
Long-time behavior for a class of Feller processes, \emph{Trans. Amer. Math. Soc.} \textbf{368},
1871-1910, 2016

\bibitem{S}
K. Sato: \emph{L{\'e}vy Processes and Infinitely Divisible Distributions}, Cambridge University Press, 1999

\bibitem{Sch}
R. Schilling:
Growth and H\"older conditions for the sample paths of Feller processes, \emph{Probab. Theory Relat. Fields}
\textbf{112}, 565-611, 1998


\bibitem{bib:ESen}
E. Seneta:
\emph{Regularly Varying Functions}, Lecture Notes in Mathematics \textbf{508}, Springer, ...

\bibitem{Sh}
Y. Shiozawa: Escape rate of symmetric jump-diffusion processes, \emph{Trans. Amer. Math. Soc.} \textbf{368},
7645-7680, 2016

\bibitem{SW}
Y. Shiozawa, J. Wang:
Rate functions for symmetric Markov processes via heat kernel, \emph{Potential Anal.} \textbf{46}, 23-53, 2017


\bibitem{S2}
B. Simon: \emph{Functional Integration and Quantum Physics}, 2nd ed., AMS Chelsea Publishing, 2004

\bibitem{ShN}
T. Sirao, T. Nisida: On some asymptotic properties concerning Brownian motion, \emph{Nagoya Math. J.}
\textbf{4}, 97-101, 1952

\bibitem{Sh}
T. Sirao: On some asymptotic properties concerning homogenous differential processes, \emph{Nagoya Math. J.}
\textbf{6}, 95-107, 1953

\bibitem{Z}
V.M. Zolotarev: Analog of the iterated logarithm law for semi-continuous stable processes, \emph{Theory Probab.
Appl.} \textbf{9}, 512-513, 1964


\iffalse
\bibitem{bib:B}
R. Ba\~nuelos:
\emph{Intrinsic ultracontractivity and eigenfunction estimates for Schr\"odinger operators},
J. Funct. Anal. 100, 1991, 181-206.

\bibitem{bib:BD}
R. Ba\~nuelos, B. Davis:
\emph{A geometrical characterization of intrinsic ultracontractivity for planar domains with
boundaries given by the graphs of functions},
Indiana Univ. Math. J. 41 (4), 1992, 885-913.

\bibitem{bib:Ber}
J. Bertoin:
\emph{L\'evy Processes},
Cambridge Univ. Press, Cambridge, 1996.

\bibitem{bib:BlG}
R. M. Blumenthal, R. K. Getoor:
\emph{Markov Processes and Potential Theory},
Springer, New York, 1968.

\bibitem{bib:BlGR}
R. M. Blumenthal, R. K. Getoor, D. B. Ray:
\emph{On the distribution of first hits for the symmetric stable processes},
Trans. Amer. Math. Soc. 99, 1961, 540-554.

\bibitem{bib:BB1}
K. Bogdan, T. Byczkowski:
\emph{Potential theory for the $\alpha$-stable Schr\"odinger operator on bounded
Lipschitz domain},
Studia Math. 133 (1), 1999, 53-92.

\bibitem{bib:BB2}
K. Bogdan, T. Byczkowski:
\emph{Potential theory of Schr\"odinger operator based on fractional Laplacian},
Prob. Math. Statist. 20, 2000, 293-335.

\bibitem{bib:BBKRSV}
K. Bogdan et al:
\emph{Potential Analysis of Stable processes and its Extensions} (ed. P. Graczyk, A. St\'os),
Lecture Notes in Mathematics 1980, Springer, Berlin, 2009.

\bibitem{bib:BKK}
K. Bogdan, T. Kumagai, M. Kwa{\'s}nicki:
\emph{Boundary Harnack inequality for Markov processes with jumps},
preprint, 2012.

%\bibitem{bib:Ca}
%R. Carmona: \emph{Pointwise bounds for Schr\"odinger eigenstates},
%Commun. Math. Phys. 62, 1978, 65-92

%\bibitem{bib:C}
%R. Carmona:
%\emph{Path integrals for relativistic Schr\"odinger operators},
%Lect. Notes in Phys. 345, 1989, 65-92.


\bibitem{bib:CMS}
R. Carmona, W. C. Masters, B. Simon:
\emph{Relativistic Schr\"odinger operators: asymptotic behaviour of the eigenfunctions},
J. Funct. Anal. 91, 1990, 117-142.

\bibitem{bib:CS}
Z. Chen, R. Song:
\emph{General gauge and conditional gauge theorems},  Ann. Probab.  30  (2002),  no. 3, 1313-1339.

\bibitem{bib:CS1}
Z. Chen, R. Song:
\emph{Intrinsic ultracontractivity and conditional gauge for symmetric stable processes},
J. Funct. Anal. 150 (1), 1997, 204-239.

\bibitem{bib:CS2}
Z. Chen, R. Song:
\emph{Intrinsic ultracontractivity, conditional lifetimes and conditional gauge for symmetric
stable processes on rough domains},
Illinois J. Math. 44 (1), 2000, 138-160.

\bibitem{bib:CZ}
K. L. Chung, Z. Zhao:
\emph{From Brownian Motion to Schr\"odinger's Equation},
Springer, New York, 1995.

\bibitem{bib:D}
E. B. Davies:
\emph{Heat Kernels and Spectral Theory },
Cambridge Univ. Press., Cambridge, 1989.

\bibitem{bib:DS}
E. B. Davies, B. Simon:
\emph{Ultracontractivity and heat kernels for Schr\"odingers operators and Dirichlet Laplacians},
J. Funct. Anal. 59, 1984, 335-395.

%\bibitem{bib:Da}
%B. Davis:
%\emph{Intrinsic ultracontractivity and the Dirichlet Laplacian},
%J. Funct. Anal. 100 (1), 1984, 162-180.

\bibitem{bib:HIL}
F. Hiroshima, T. Ichinose, J. L\H{o}rinczi: \emph{Path integral representation for Schr\"odinger
operators with Bernstein functions of the Laplacian}, arXiv 0906.0103, 2009.

\bibitem{bib:HL}
F. Hiroshima, J. L\H{o}rinczi: \emph{Functional integral representation of the Pauli-Fierz
model with spin $1/2$}, J. Funct. Anal. 254, 2008, 2127-2185.

\bibitem{bib:KaKu}
K. Kaleta, T. Kulczycki:
\emph{Intrinsic ultracontractivity for Schr\"odinger operators based on fractional Laplacians},
Potential Anal. 33 (4), 2010, 313-339.

%\bibitem{bib:Ku1}
%T. Kulczycki:
%\emph{Properties of Green function of symmetric stable processes},
%Probab. Math. Statist. 17, 1997, 339-364.

\bibitem{bib:Ku2}
T. Kulczycki:
\emph{Intrinsic ultracontractivity for symmetric stable process},
Bull. Polish Acad. Sci. Math. 46 (3), 1998, 325-334.

\bibitem{bib:KS}
T. Kulczycki, B. Siudeja:
\emph{Intrinsic ultracontractivity of the Feynman-Kac semigroup for the relativistic stable process},
Trans. Amer. Math. Soc. 358 (11), 2006, 5025-5057.

\bibitem{bib:Kw}
M. Kwa{\'s}nicki:
\emph{Intrinsic ultracontractivity for stable semigroups on unbounded open sets},
Potential Anal. 31 (1), 2009, 57-77.

%\bibitem{bib:La}
%N. S. Landkof:
%\emph{Foundations of Modern Potential Theory},
%Springer, New York, 1972.

\bibitem{bib:LMa}
J. L\H{o}rinczi, J. Ma{\l}ecki:
\emph{Spectral properties of the massless relativistic harmonic oscillator},
arXiv:1006.3665, 2010.

\bibitem{bib:LHB}
J. L\H{o}rinczi, F. Hiroshima, V. Betz: \emph{Feynman-Kac-Type
Theorems and Gibbs Measures on Path Space. With Applications to
Rigorous Quantum Field Theory}, de Gruyter Studies in Mathematics
\textbf{34}, Walter de Gruyter, Berlin-New York, to appear, 2010

\bibitem{bib:RS}
M. Reed and B. Simon:
\emph{Methods of Modern Mathematical Physics, 1. Functional Analysis}, Academic Press, 1980.

%\bibitem{bib:R}
%M. Ryznar:
%\emph{Estimates of Green function for relativistic $\alpha$-stable process},
%Potential Anal. 17, 2002, 1-23.

%\bibitem{bib:SC}
%L. Saloff-Coste:
%\emph{Aspects of Sobolev-Type Inequalities}, Cambridge University Press, 2001.

\bibitem{bib:R}
P. R\'ev\'esz: \emph{Random Walk in Random and Non-Random Environments}, World Scientific, 2005




\bibitem{bib:S}
B. Simon:
\emph{Schr\"odinger semigroups},
Bull. Amer. Math. Soc. 7 (3), 1982, 447-526.

\bibitem{bib:S2}
B. Simon:
\emph{Functional Integration and Quantum Physics},
2nd ed., AMS Chelsea Publishing, 2004.

%\bibitem{bib:SW}
%R. Song, J.-M. Wu:
%\emph{Boundary Harnack principle for symmetric stable processes},
%J. Funct. Anal. 168, 1999, 403-427.

\bibitem{bib:Szt}
P. Sztonyk:
\emph{Transition density estimates for jump Lévy processes}
Stochastic Process. Appl. 121, 2011, 1245-1265.

\bibitem{bib:Z}
Z. Zhao:
\emph{A probabilistic principle and generalized Schr\"odinger perturbations},
J. Funct. Anal. 101 (1), 1991, 162-176.


\fi

\end{thebibliography}
\end{document}